\documentclass[11pt, reqno]{amsart}
\usepackage{amssymb, hyperref}
\usepackage{amsthm}
\usepackage{tikz}
\usepackage{tikz}
\usepackage{amsmath}
\usetikzlibrary{arrows.meta}
\usepackage[all,2cell]{xy}
\usepackage{verbatim}
\usepackage{xcolor,soul}
\usepackage{hyperref}
\usepackage{mathrsfs}
\usepackage{csquotes}
\usepackage{epigraph}
\numberwithin{equation}{section}
\usepackage[pagewise,mathlines]{lineno}

\newtheorem{theorem}{Theorem}[section]
\newtheorem{corollary}[theorem]{Corollary}
\newtheorem{definition}[theorem]{Definition}
\newtheorem{lemma}[theorem]{Lemma}

\newtheorem{proposition}[theorem]{Proposition}

\theoremstyle{remark}
\newtheorem{remark}[theorem]{Remark}
\newtheorem*{theorem*}{Theorem}
\newtheorem*{corollary*}{\bf Corollary}

\theoremstyle{remark}
\newtheorem{example}[theorem]{Example}

\title[The Gromov width of generalized Bott-Samelson manifolds]{The Gromov width of generalized Bott-Samelson manifolds}

\author[B.~N.~Chary]{Narasimha Chary Bonala}
\address{Narasimha Chary Bonala\\
Department of Mathematics, Indian Institute of Technology Kanpur,
U.P. India, 208016.
}
\email{chary@iitk.ac.in}

\author[Y. Singh]{Yogendra Singh}
\address{Yogendra Singh\\ Department of Mathematics, Indian Institute of Technology Kanpur,
U.P. India, 208016.}
\email{ysingh23@iitk.ac.in}

\thanks{}



\begin{document}


\begin{abstract}
We study the Gromov width of smooth generalized Bott-Samelson varieties, a class of projective varieties constructed by Perrin in \cite{Per07} as a generalization of classical Bott-Samelson resolutions of Schubert varieties. We show that the Gromov width of such a variety equipped with a rational K\"ahler form is given by the symplectic area of its minimal rational curves. As a consequence, we obtain upper bounds for the Seshadri constants of these varieties with respect to ample line bundles.
\end{abstract}


\maketitle
\section{Introduction}

The \emph{Gromov width} $w_G(M, \omega)$ of a $2n$-dimensional symplectic manifold $(M, \omega)$ is a symplectic invariant that measures the size of the largest standard symplectic ball that can be symplectically embedded into $(M, \omega)$. More precisely, it is defined as 
\[
w_G(M, \omega) := sup\{ a: (B^{2n}(\sqrt{a/\pi}),\omega_{st})~\text{symplectically embeds into}~ (M,\omega)\},
\]
where $(B^{2n}(r), \omega_{st})$ is the ball of radius $r$ centered at the origin in $\mathbb R^{2n}$ equipped with the standard symplectic form.
Darboux’s theorem ensures that this symplectic invariant is strictly positive for all symplectic manifolds. The concept of Gromov width originated in Gromov’s foundational work on symplectic rigidity, where it played a central role in the proof of the non-squeezing theorem \cite{Gro85}. Since then, the Gromov width has emerged as a fundamental invariant in symplectic geometry, and substantial effort has been devoted to computing or bounding it in various geometric settings; see, for instance, \cite{Bir01, KT05, LMZ15, Cas16, FLP18, HLS21, BCF24a, BCF24b} and the references therein.

In this paper, we estimate the Gromov width of a certain class of smooth complex projective varieties, known as \emph{smooth generalized Bott-Samelson varieties} $\widehat{X}(\widehat{w})$, introduced by Perrin in \cite{Per07}.
We briefly recall them now. A distinguished class of subvarieties of flag varieties is known as \emph{Schubert varieties}, which play a central role in the geometry of flag varieties. In general, Schubert varieties are not smooth, but they admit natural desingularizations called \emph{Bott-Samelson varieties}. These were originally introduced by Bott and Samelson \cite{BS58} in a differential-geometric and topological settings, and later adapted independently by Demazure \cite{Dem74} and Hansen \cite{Han73} to the algebro-geometric context. A Bott-Samelson variety $Z(\tilde{w})$ can be viewed as an iterated $\mathbb{P}^1$-bundle associated with a reduced expression $\tilde{w}$ of an element $w$ in the Weyl group. The geometry of these varieties depends very much on the chosen expression of the Weyl group element corresponding to the Schubert variety (see, for example, \cite[Page 32]{BKP15} and \cite{BK17}).
 The Bott-Samelson resolutions are rarely small: fibers often have large dimensions, and automorphisms of the target need not lift to the resolution (see \cite[§7]{BKP15} for explicit examples).
In contrast, \emph{generalized} Bott-Samelson varieties include small resolutions of Schubert varieties constructed by Zelevinsky \cite{Zel83}, by Sankaran and Vanchinathan \cite{SV94, SV95}, and by Perrin in full generality (see \cite[Cor.~7.9]{Per07}). They are constructed as iterated fibrations over Schubert varieties, with fibers themselves being isomorphic to Schubert varieties. In this paper, we consider \emph{smooth generalized Bott-Samelson varieties} $\widehat{X}(\widehat{w})$ associated with admissible generalized reduced decompositions of Weyl group elements $w$; see Section~\ref{GBSV} for the definition and further details.

Our main result establishes that the Gromov width of a smooth generalized Bott-Samelson variety coincides with the symplectic area of its minimal rational curves. This provides a direct and computable formula for the Gromov width, extending earlier results of \cite{BCF24a} for Bott-Samelson varieties. Gromov widths are closely related to Seshadri constants, the key invariants of line bundles in algebraic geometry that measure local positivity (see \cite[Chapter 5]{Laz04} for more details). Our computation of the Gromov width yields estimates for the Seshadri constants of ample line bundles on smooth generalized Bott-Samelson varieties.

Our approach proceeds in two parts. First, using \cite[Theorem~1.1]{BCF24b}, we derive an upper bound for the Gromov width in terms of the symplectic area of minimal rational curves (see Corollary \ref{cor:upperbound}). Using a result of Brion and Kannan~\cite{BK21} on the classification of $T$-stable curves on $\widehat{X}(\widehat{w})$, we obtain a bijection between the $T$-stable curves of $\widehat{X}(\widehat{w})$ and those of $Z(\tilde{w})$ (see Lemma~\ref{lem:curves}). Moreover, in~\cite{BK21} they also characterize the minimal rational curves in this setting under certain assumptions. Building on these results, we derive the desired upper bound.

To obtain a matching lower bound, we employ symplectic embedding techniques, exploiting toric degenerations and Newton-Okounkov bodies. In particular, we construct explicit symplectic embeddings of balls into appropriate momentum polytopes associated with these varieties. This method generalizes to broader classes of projective varieties and has been effectively used by several authors; see, for instance, \cite{KK19} for a survey. Our construction uses unimodular simplices contained in Newton-Okounkov bodies, which in our case coincide with generalized string polytopes as described in \cite{Fujita} (also see \cite[Remark 4.2]{BCF24a}). We note that the Newton-Okounkov bodies of the generalized Bott-Samelson variety $\widehat{X}(\widehat{w})$ with respect to a very ample line bundle $\mathcal{L}$ are identical to those of the Bott-Samelson variety $Z(\tilde{w})$ for the pullback line bundle $\mathcal{L}'$ of $\mathcal L$ (see Lemma~\ref{lem:NObody}). This allows us to adapt the techniques of \cite{BCF24a} to our setting, thereby constructing a simplex whose size matches the minimal symplectic area of the minimal rational curves in $\widehat{X}(\widehat{w})$. This leads to our main theorem.
\begin{theorem}\label{thm:main}    
Let $\widehat{X}(\widehat{w})$ be a smooth generalized Bott-Samelson variety equipped with a rational K\"ahler $2$-form $\omega$. Then the Gromov width is given by
\[
w_{G}(\widehat{X}(\widehat{w}), \omega) = \min\left\{\, \int_C \omega \;\middle|\; C \text{ is a minimal rational curve in }
\widehat{X}(\widehat{w})\, \right\}.
\]
\end{theorem}
To describe rational K\"ahler $2$-forms on $\widehat{X}(\widehat{w})$, we first study line bundles on it. In \cite{LT}, Lauritzen and Thomsen provided an explicit basis for the Picard group of the classical Bott-Samelson variety $Z(\tilde w)$, along with criteria for a line bundle to be (very) ample. In Section~\ref{sec:line}, we extend these results to the smooth generalized Bott-Samelson varieties $\widehat{X}(\widehat{w})$: we construct a basis for their Picard groups (Theorem~\ref{thm:linebundles}) and establish criteria for a line bundle to be very ample or nef (Theorem~\ref{thm:very ample}). From this classification, we also notice that any ample line bundle over $\widehat{X}(\widehat{w})$ is very ample (Corollary \ref{cor:ample}).

Using the description of $T$-invariant curves on $\widehat{X}(\widehat{w})$ together with the characterization of free rational curves (see Lemma~\ref{lem:free}), we can explicitly determine the Gromov width of $(\widehat{X}(\widehat{w}),\omega)$ for a given rational K\"ahler $2$-form $\omega$ on $\widehat{X}(\widehat{w})$; see Example~\ref{ex:Gromov width}.

We conclude the introduction by giving upper bounds for the Seshadri constants of generalized Bott-Samelson varieties. Given a projective variety $X$ together with a nef line bundle $\mathcal{L}$, the Seshadri constant
$\varepsilon(X,\mathcal{L},x)$ of $\mathcal{L}$ at a point $x \in X$ is defined as the infimum of the ratio
$\mathcal{L}\cdot C / \operatorname{mult}_x C$, taken over all irreducible and reduced curves $C$ on $X$
passing through $x$. Here, $\operatorname{mult}_x C$ denotes the multiplicity of $C$ at $x$.

For a projective complex manifold 
$X$ equipped with a very ample line bundle, \cite[Proposition 6.2.1]{BC01} shows that the Seshadri constant is always bounded above by the Gromov width of $X$ equipped with the Fubini-Study form $\omega_{\mathcal L}$ associated
to $\mathcal L$. Therefore, Theorem~\ref{thm:main} yields upper bounds for the Seshadri constants of generalized Bott-Samelson varieties, as recorded in Corollary~\ref{Cor:Ses}.
\begin{corollary}
    Let $\widehat{X}(\widehat{w})$ be a smooth generalized Bott-Samelson variety equipped with an ample line bundle $\mathcal{L}$. Then for any $x \in \widehat{X}(\widehat{w})$,
    \[
        \varepsilon(\widehat{X}(\widehat{w}), \mathcal{L}, x) \leq 
        \min\left\{\, \mathcal{L} \cdot C \;\middle|\; 
        C \text{ is a minimal rational curve in } 
        \widehat{X}(\widehat{w})\, \right\}.
    \]
\end{corollary}

\subsection{Structure of the paper}
The paper is organized as follows. In Section~\ref{sec:basics}, we review basic notions and definitions. In Section~\ref{GBSV}, we recall the construction of generalized Bott-Samelson varieties and describe their relation to Bott-Samelson varieties. In Section~\ref{sec:line}, we give a description of line bundles on generalized Bott-Samelson varieties and provide a characterization of very ample and nef line bundles. In Section~\ref{sec:UB}, we review the notion of minimal rational curves, describe free rational curves on generalized Bott-Samelson varieties, and obtain an upper bound for the Gromov width in terms of minimal rational curves. In Section~\ref{sec:LB}, we recall the notion of Newton-Okounkov bodies, derive a lower bound for the Gromov width, and prove Theorem~\ref{thm:main}. Section~\ref{sec:Seshadri con} discusses the connection with Seshadri constants. In Appendix~\ref{f_0}, we consider Newton-Okounkov bodies for nef line bundles on Bott-Samelson varieties and prove that the simplex of the required size can be embedded in them.


\section{Preliminaries}\label{sec:basics}

In this section, we briefly setup the notation and review relevant definitions. For the necessary background on linear algebraic groups and Schubert varieties, we refer to \cite{BK}, \cite{Hum2}, \cite{Bor}, and \cite{Spr}.

 Let $G$ be a simply-connected semisimple complex algebraic group of rank $n$. Fix a maximal torus $T \subset G$ and a Borel subgroup $B \subset G$ containing $T$.
Let $R$ denote the root system of $(G, T)$ and let $R^+ \subset R$ be the subset of positive roots corresponding to $(B, T)$. Then $R^+$ forms a system of positive roots of $R$. Let $R^- = -R^+$ be the corresponding set of negative roots, and let
\[
S = \{ \alpha_1, \ldots, \alpha_n \} \subset R^+
\]
be the set of simple roots. The Weyl group $W = N_G(T)/T$ is generated by the simple reflections $s_1, \ldots, s_n$ corresponding to the simple roots.

For any root $\beta \in R$, denote by $U_\beta \subset G$ the associated root subgroup and by $G_\beta \subset G$ the subgroup generated by $U_\beta$ and $U_{-\beta}$. For any simple root $\alpha\in S$, we denote by $P^{\alpha}$ and $P_{\alpha}$ the maximal and minimal parabolic subgroups of $G$, respectively.
We also consider the coroot system $R^\vee$, whose simple coroots $\alpha_1^\vee, \ldots, \alpha_n^\vee$ form a basis of the cocharacter lattice $\Xi_*(T)$. The dual basis of the character lattice $\Xi^*(T)$ consists of the fundamental weights $\omega_1, \ldots, \omega_n$. More intrinsically, for any simple root $\alpha$, we denote by $\omega_\alpha$ the fundamental weight defined by
\[
{\omega}_\alpha(\beta^\vee) =
\begin{cases}
1 & \text{if } \beta = \alpha, \\
0 & \text{if } \beta \neq \alpha,
\end{cases}
\]
for all simple roots $\beta$.

Let $X$ be a projective variety equipped with a transitive action of $G$. Choose a point $x \in X$ and assume that the stabilizer subgroup $G_x$ is smooth. Then $X$ can be identified with the homogeneous space $G/P$, where $P := G_x$ is a connected parabolic subgroup (see, for example, \cite[Chapter IV, Section 11]{Bor}. Under this identification, the points $x \in X$ and $eP \in G/P$ are both referred to as base points.

Given a subset $I \subset S$, let $P_I$ denote the standard parabolic subgroup of $G$ corresponding to $I$. Let $W_I$ be the subgroup of $W$ generated by the simple reflections $s_\alpha$ for $\alpha \in I$. Then
\[
W_I = N_{L_I}(T)/T,
\]
where $L_I$ is a Levi factor of $P_I$. Define $W^I$ as the subset of $W$ consisting of elements $w$ such that $w(\alpha) \in R^+$ for all $\alpha \in I$; this subset is often denoted $W^{P_I}$ in the literature. Then $W^I$ forms a set of minimal-length representatives of the cosets in $W/W_I$.

For any $w \in W$, the \textit{support} of $w$, denoted $\operatorname{Supp}(w)$, is the set of simple roots that appear in a reduced expression for $w$. Let $G_w$ denote the subgroup of $G$ generated by the root subgroups $U_{\pm \alpha}$ for $\alpha \in \operatorname{Supp}(w)$. Then $G_w$ is the derived subgroup of the Levi subgroup $L_{\operatorname{Supp}(w)}$ and is therefore semisimple, normalized by $T$ and contains a representative of $w$.

Let $P^w$ be the largest parabolic subgroup of $G$ such that $B \subset P^w$ and $w \in W^{P^w}$. Then
\[
P^w = P_{I^w}, \quad \text{where } I^w := \{ \alpha \in S \mid w(\alpha) \in R^+ \}.
\]
Equivalently, $P^w$ is generated by $B$ and $U_{-\alpha}$ with $\alpha\in I^{w}$. 
Consider the base point $x = eP^w \in G/P^w$ and the point $wx \in G/P^w$. The $B$-orbit $Bwx$ is a locally closed subvariety of $G/P^w$ and the associated \textit{Schubert variety} $X(w)$ is defined as its Zariski closure:
\[
X(w) := \overline{Bwx} \subset G/P^w.
\]

Let $P_w$ be the subgroup of $G$ stabilizing $X(w)$, i.e.,
\[
P_w := \{ g \in G \mid g X(w) = X(w) \}.
\]
Then $P_w$ is a parabolic subgroup containing $B$ and hence of the form $P_{I_w}$, where
\[
I_w := \{ \alpha \in S \mid s_\alpha w \leq w \}
\]
with respect to the Bruhat order on $W^{I^w} = W/W_{I^w}$. Finally, observe that $P_w \cap G_w$ is a parabolic subgroup of $G_w$ and we have the isomorphism
\[
X(w) = \overline{P_w w P^w}/P^w \simeq \overline{(P_w \cap G_w) w (P^w \cap G_w)}/(P^w \cap G_w) \subset G_w / (P^w \cap G_w).
\]

\section{Generalized Bott-Samelson Varieties}\label{GBSV}

In this section, we outline a generalization of the classical Bott-Samelson construction, based on suitable factorizations in the Weyl group. In this section, we follow \cite{Per07}.

We first recall the following action:
Let an algebraic group $H$ acts on varieties $X$ and $Y$. The \emph{twisted fiber product} $X \times^H Y$ is defined as the quotient of $X \times Y$ by the relation
\[
(x, y) \sim (xh, h^{-1}y), \quad \text{for all } h \in H.
\]
For a given $w\in W$, let us denote 
\[
\widehat{X}(w):=\overline{(P_w \cap G_w) w (P^w \cap G_w)}/(P^w \cap G_w).
\]
Note that $\widehat{X}(w)\simeq X(w)$. 

Let $w = w_1w_2 \in W$ with $P^{w_1} \cap G_{w_1} \subset P_{w_2}$. Define
\[
\widehat{X}(w_1,w_2) := \overline{(P_{w_1} \cap G_{w_1}) w_1 (P^{w_1} \cap G_{w_1})} \times^{P^{w_1} \cap G_{w_1}} \widehat{X}(w_2),
\]
where $\widehat{X}(w_2) \simeq X(w_2)$.
This variety is projective and admits a Zariski locally-trivial fibration
\[
f_{w_1,w_2} : \widehat{X}(w_1,w_2) \to \widehat{X}(w_1), \quad [p_1, p_2] \mapsto [p_1],
\]
with fiber $\widehat{X}(w_2)$ and is equivariant under $P_{w_1} \cap G_{w_1}$ acting on the first component.
There is also a morphism
\[
\pi_{w_1,w_2} : \widehat{X}(w_1,w_2) \to G/P^w, \quad [p_1, p_2] \mapsto p_1p_2P^w.
\]
This construction extends inductively under suitable conditions.

\begin{definition}\label{def:gBS}
Define:

\begin{enumerate} 
    \item 
A sequence $\widehat{w} = (w_1, \dots, w_m)$ in $W$ is a \emph{generalized decomposition} of $w \in W$ if $w = w_1 \cdots w_m$.

\item For any generalized decomposition $\widehat{w}$ of $w$, we inductively define a sequence of parabolic subgroups $(P_{J_i})_{1\leq i\leq m}$ as follows:
\[\quad {J_m}:=I_{w_m}\quad \text{and} \quad J_i:=( J_{{i+1}}\cup\operatorname{Supp}(w_i))\cap( w_{i}^{\perp} \cup \operatorname{Supp}(w_i))\cap (I_{w_i}\cup \operatorname{Supp}(w_i)^c),\]
where
\[\operatorname{Supp}(w_i)^c:=S\setminus \operatorname{Supp}(w_i)\quad \text{and} \quad
w_i^\perp := \{ \alpha \in S \mid w_i s_\alpha = s_\alpha w_i  \}.
\]

\item A generalized decomposition $\widehat{w}$ of $w$ is \emph{admissible} if, for all $1\leq i<m$,
\[
P^{w_i}\cap G_{w_i}\subset P_{J_{i+1}}, \quad\text{equivalently}\quad I^{w_i}\cap \operatorname{Supp}(w_i)\subset J_{i+1}.
\]

\item 
 A generalized decomposition $\widehat{w}$ of $w$ is \emph{good} if for all $1 \le i < m$, we have
\[
P^{w_i} \cap G_{w_i} \subset P_{w_{i+1}\cdots w_m}
\quad \text{and} \quad
 I_{w_i \cdots w_m} \subset w_i^\perp \cup \operatorname{Supp}(w_i).
\]
\end{enumerate}

\end{definition}

\begin{remark}\

\begin{enumerate}
\item The condition $P^{w_i} \cap G_{w_i} \subset P_{w_{i+1}\cdots w_m}$ given in  Definition \ref{def:gBS}(4), is equivalent to \[ I^{w_i}\cap \operatorname{Supp}(w_i)\subset I_{w_{i+1}\cdots w_m}. \]
    \item    
If $\widehat{w} = (w_1,\dots,w_m)$ is an admissible (resp. good) generalized decomposition of $w$, then each tail $(w_{i+1},\dots,w_m)$ is also an admissible (resp. good) generalized decomposition of $w_{i+1}\cdots w_m$.
\item  
In the minuscule case, the second condition 
\[
I_{w_i\cdots w_m}\subset w_i^{\perp}\cup \operatorname{Supp}(w_i)
\]
in the definition of good generalized decompositions ensures that the parabolic subgroups $P_{J_i}$ have the nice structures. In that case, precisely we have the following 
\[
P_{J_i}=P_{w_i\cdots w_m}.
\] As a consequence, in the minuscule case, every good generalized decomposition is admissible (see for more details \cite[Proposition 5.3]{Per07}).
\item In general, admissible and good are two different notions. See the following two examples:
\begin{enumerate}
    \item \emph{Admissible but not good}:
 Let $G=SL_4(\mathbb{C})$ and consider the generalized reduced decomposition \[\widehat{w}=(w_1,w_2,w_3)=(s_2,s_1,s_2)\] of $w=s_2s_1s_2\in W$. We have $J_3=\{\alpha_2\}$, $J_2=\{\alpha_1\}$, and $J_1=\{\alpha_2\}$. Since $\emptyset\subset J_i$ for all $1\le i\le 3$, the decomposition $\widehat{w}$ is admissible. However, $I_{w_1w_2w_3}=\{\alpha_1,\alpha_2\}$, whereas $w_1^{\perp}\cup \operatorname{Supp}(w_1)=\{\alpha_2\}$. Therefore, $I_{w_1w_2w_3}\not\subset w_1^{\perp}\cup \operatorname{Supp}(w_1)$, showing that $\widehat{w}$ is not a good generalized decomposition of $w$.

\item \emph{Good but not admissible}: Let $G=SL_4(\mathbb{C})$ and consider the generalized reduced decomposition 
\[
\widehat{w}=(w_1,w_2,w_3)=(s_3s_2s_1,s_3s_2,s_3)
\]
of the longest element $w=s_3s_2s_1s_3s_2s_3\in W$. We have \[J_2=\{\alpha_3\}\quad \text{and}\quad I^{w_1}\cap \operatorname{Supp}(w_1)=\{\alpha_2,\alpha_3\}.\]
Clearly, $I^{w_1}\cap \operatorname{Supp}(w_1)\not\subset J_2$, thus $\widehat{w}$ is not admissible. But $\widehat{w}$ is good, as we have
\[I_{w_1w_2w_3}=\{\alpha_1,\alpha_2,\alpha_3\}, \quad
I_{w_2w_3}=\{\alpha_2,\alpha_3\},\quad I_{w_3}=\{\alpha_3\}\]
and 
\[I^{w_1}\cap \operatorname{Supp}(w_1)=\{\alpha_2,\alpha_3\},\quad I^{w_2}\cap \operatorname{Supp}(w_2)=\{\alpha_3\}.
\]
Also, we have $w_1^{\perp}\cup\operatorname{Supp}(w_1)=\{\alpha_1,\alpha_2,\alpha_3\}$ and $w_2^{\perp}\cup \operatorname{Supp}(w_2)=\{\alpha_2,\alpha_3\}.$ All these together satisfy the conditions for Definition \ref{def:gBS}(4) as follows
\[
I^{w_1}\cap \operatorname{Supp}(w_1)\subset I_{w_2w_3},\quad I_{w_1w_2w_3}\subset w_1^{\perp}\cup \operatorname{Supp}(w_1)
\]and\[
I^{w_2}\cap \operatorname{Supp}(w_2)\subset I_{w_3},\quad I_{w_2w_3}\subset w_2^{\perp}\cup \operatorname{Supp}(w_2).
\]
Hence, $\widehat{w}$ is good but not admissible.
\end{enumerate}
    \end{enumerate}
\end{remark}

For any $1\leq i<m$, since
$
I_{w_{i+1}}\cap\operatorname{Supp}(w_{i+1})\subset J_{i+1},
$
the action of $P_{w_{i+1}}\cap G_{w_{i+1}}$ on $\widehat{X}(w_{i+1},\dots,w_m)$ extends to an action of $P_{J_{i+1}}$ on $\widehat{X}(w_{i+1},\dots,w_m)$; see \cite[Lemma~5.1]{Per07}. The condition $P^{w_i}\cap G_{w_i}\subset P_{J_{i+1}}$ in the definition of admissibility ensures that there exists a well-defined twisted fiber product
\[
\widehat{X}(w_i,\dots,w_m):=\overline{(P_{w_i}\cap G_{w_i})w_i(P^{w_i}\cap G_{w_i})}\times^{P^{w_i}\cap G_{w_i}} \widehat{X}(w_{i+1},\dots,w_{m}).
\]
Thus, we arrive at the following definition.

\begin{definition}
Let $w\in W$ and $\widehat{w} = (w_1, \dots, w_m)$ be an admissible generalized decomposition of $w$. The associated projective variety $\widehat{X}(\widehat{w})$ is defined recursively and admits a Zariski locally trivial fibration
\[
\widehat{f} : \widehat{X}(\widehat{w}) \to \widehat{X}(w_1), \quad [p_1,\dots,p_m] \mapsto [p_1],
\]
with fiber $\widehat{X}(w_2,\dots,w_m)$. This variety is called the \emph{generalized Bott--Samelson variety}.
\end{definition}

There is a $(P_{w_1} \cap G_{w_1})$-equivariant morphism
\begin{equation}\label{eq:pi map}
\widehat{\pi}: \widehat{X}(\widehat{w}) \to {X}(w), \quad [p_1,\dots,p_m] \mapsto p_1 \cdots p_m P^w.
\end{equation}

\begin{definition}
An admissible generalized decomposition $\widehat{w}$ is \emph{reduced} if $\ell(w) = \sum_{j=1}^m \ell(w_j)$.
\end{definition}

 Throughout this paper, we assume that $\widehat{w}$ is an admissible generalized reduced decomposition unless stated otherwise.

If $\widehat{w}$ is reduced, then $\widehat{\pi}$ is birational onto its image $X(w)$; it restricts to an isomorphism over the open $B$-orbit $Bwx$ and maps the base point $\widehat{x} = [w_1,\dots,w_m]$ to $wx$. 

The map $\widehat{f}$ sends $\widehat{x}$ to $w_1x_1$, where $x_1$ is the base point of $G_{w_1}/(P^{w_1} \cap G_{w_1})$. The $B$-orbit $B\widehat{x}$ maps isomorphically onto $Bwx$ under $\widehat{\pi}$.
More generally, for $1 \le j \le m$, define
\[
\widehat{f}_j : \widehat{X}(\widehat{w}) \to \widehat{X}(w_1,\dots,w_j), \quad [p_1,\dots,p_m] \mapsto [p_1,\dots,p_j],
\]
with fiber $\widehat{X}(w_{j+1},\dots,w_m)$ and $\widehat{f}_1 = \widehat{f}$. By \cite[Lemma~5.1]{Per07}, the parabolic subgroup $P_{J_1}$ acts on $\widehat{X}(\widehat{w})$ and recall that $P_{w_1}$ is the stabilizer of $X(w_1)$, with these actions we note that the morphism $\widehat{f}$ is $B$-equivariant. 

Also, note that  the stabilizer of the base point $\widehat{x}$ 
contains the maximal torus $T$.
\begin{remark}\label{rmk:smoothness}Since smoothness is a local property and the morphism
$
\widehat{f} : \widehat{X}(\widehat{w}) \to \widehat{X}(w_1)
$
is a locally trivial fibration with fiber $\widehat{X}(w_2, \dots, w_m)$,
it follows, by induction on $m$, that the generalized Bott-Samelson variety $\widehat{X}(\widehat{w})$ is smooth if and only if the Schubert variety $X(w_i)$ is smooth for every $i = 1, \dots, m$.
 \end{remark}


\subsection{Link with Bott-Samelson Varieties}\label{sec:linktoBS}

In this subsection, we first recall Bott-Samelson variety  associated with an expression of $w\in W$ in terms of simple reflections. Also, we notice that, for a given generalized Bott-Samelson variety $\widehat{X}(\widehat{w})$ with respect to an admissible generalized reduced decomposition, we get a Bott-Samelson variety as a desingularization of $\widehat{X}(\widehat{w})$.

Let $\tilde{w}=(s_{i_1},\dots, s_{i_r})$ be a sequence, corresponding to an expression of $w= s_{i_1}\cdots s_{i_r}\in W$ (need not be reduced) in terms of simple reflections. Then the associated \textit{Bott-Samelson variety} $Z(\tilde{w})$ is the quotient 
\begin{center}
    $Z(\tilde{w}):=\big(P_{\alpha_{i_1}}\times P_{\alpha_{i_2}}\times\cdots\times P_{\alpha_{i_r}}\big)/B^{r}$
\end{center}
where $B^r$ acts from right on $P_{\alpha_{i_1}}\times P_{\alpha_{i_2}}\times\cdots\times P_{\alpha_{i_r}}$ by \begin{center}
    $(p_1,\dots,p_r)\cdot(b_1,\dots,b_r):=(p_1b_1, b_1^{-1}p_2b_2,\dots,b_{r-1}^{-1}p_rb_r)$
\end{center}
for $p_j\in P_{\alpha_{i_j}}$ and $b_{j}\in B$ for all $1 \leq j \leq r$. Moreover, if the expression of $w$ is reduced, then the Bott-Samelson variety $Z(\tilde{w})$ is a desingularization of the Schubert variety $X(w)$ in $G/P^w$ associated to $w$. 

\begin{remark}
    We may also write $Z(\tilde{w})$ as 
    \begin{center} 
        $P_{\alpha_{i_1}}\times^B P_{\alpha_{i_2}}\times^B\cdots \times^B P_{\alpha_{i_r}}/B$.  
    \end{center}Thus, it is a twisted product of $\mathbb P^1$'s.
   \end{remark}

\begin{lemma}\label{lem:reduced decomposition}

Let $\widehat{w} = (w_1, \dots, w_m)$ be an admissible generalized reduced decomposition of $w$, and 
$w_j= \prod_{k=1}^{r_j} s_{k,j}$, where $r_j=\ell(w_j)$ for all $1\leq j\leq m$. Then the sequence
\[
\tilde{w} = (s_{1,1}, \dots, s_{r_1,1}, \dots, s_{1,m}, \dots, s_{r_m,m})
\]
is also an admissible generalized reduced decomposition of $w$, with
\begin{equation}\label{eq:reindex}
w = \prod_{j=1}^{m} \prod_{k=1}^{r_j} s_{k,j}.
\end{equation}
\end{lemma}
\begin{proof}
   This follows readily from the definition of an admissible generalized reduced decomposition.
\end{proof}

For simplicity, we reindex the expression of $w$ in Equation \ref{eq:reindex} as
\begin{center}
    $w = \prod_{j=1}^{m} \prod_{k=1}^{r_j} s_{t_{j}(k)}$,
\end{center}
where $s_{t_{j}(k)}:=s_{k,j}$ with $t_j(k) = \sum_{i=1}^{j-1} \ell(w_i) + k $ for any $1\leq j \leq m$ and $ 1\leq k \leq r_j$.

\begin{corollary}\label{cor:birational}
    The Bott-Samelson variety $Z(\tilde{w})$ is birational to the generalized Bott-Samelson variety $\widehat{X}(\widehat{w})$. In particular, $\widehat{X}(\widehat{w})$ is birational to $X(w)$.
\end{corollary}

\begin{proof}
There is a natural morphism $
\tilde{\pi}: Z(\tilde{w}) \to \widehat{X}(\widehat{w}),
$
that maps
\[
[p_{1,1}, \dots, p_{r_1,1}, \dots, p_{1,m}, \dots, p_{r_m,m}] 
\mapsto 
\left[\prod_{k=1}^{r_1} p_{k,1}, \dots, \prod_{k=1}^{r_m} p_{k,m}\right],
\] 
where $p_{k,j}\in P_{\alpha_{k,j}}$ for all $1\leq j \leq m $ and $1\leq k \leq r_j$.
The morphism $\tilde{\pi}$ sends the open $B$-stable neighbourhood $B\tilde{x}$ of the base point $\tilde{x} \in Z(\tilde{w})$ isomorphically onto the open $B$-stable neighbourhood $B\widehat{x}$ of the base point $\widehat{x} \in \widehat{X}(\widehat{w})$. Also, the open $B$-stable neighbourhood $B\widehat{x}$ maps isomorphically onto the $B$-stable open neighbourhood $Bwx$ of the base point $wx\in X(w)$ under $\widehat{\pi}$.

Now, by using Equation \ref{eq:pi map}, the composition morphism $\widehat{\pi} \circ \tilde{\pi}$ maps  $B\widehat{x}$ isomorphically onto $B\tilde{x}$. This yields the Bott-Samelson resolution
\begin{equation}\label{BS resoltion}
    \pi: Z(\tilde{w}) \to X(w).
\end{equation}
Therefore, $Z(\tilde{w})$ and $\widehat{X}(\widehat{w})$ are birational to each other.
\end{proof}

 
\section{Line bundles on \texorpdfstring{$\widehat{X}(\widehat{w})$}{X-hat(w-hat)}} \label{sec:line}

In this section, we first briefly recall the line bundles on the Bott-Samelson varieties and then extend the construction to the smooth generalized Bott-Samelson varieties.

\subsection{Line bundles on \texorpdfstring{$Z(\tilde{w})$}{Z(tilde w)}} \label{subsec:line}

Let $w \in W$ be expressed as $w = s_{i_1} \cdots s_{i_m}$. Then $\widehat{w} = (s_{i_1}, \dots, s_{i_m})$ is an admissible generalized decomposition. For $1 \leq j \leq m$, note that $P^{s_{i_j}} = P_{S \setminus \{\alpha_{i_j}\}}$ is a maximal parabolic subgroup with $G_{s_{i_j}} = G_{\alpha_{i_j}}$, and $P_{s_{i_j}} = P_{\{\alpha_{i_j}\} \cup s_{i_j}^\perp}$. Thus,
\[
P^{s_{i_j}} \cap G_{s_{i_j}} = B \cap G_{\alpha_{i_j}}
\]
is a Borel subgroup of $G_{\alpha_{i_j}}$. It follows that
\begin{align*}
\widehat{X}(s_{i_1}, \dots, s_{i_m}) &= G_{\alpha_{i_1}} \times^{B \cap G_{\alpha_{i_1}}} \widehat{X}(s_{i_2}, \dots, s_{i_m}) \\
&\simeq P_{\alpha_{i_1}} \times^B \widehat{X}(s_{i_2}, \dots, s_{i_m}) \\
&\simeq P_{\alpha_{i_1}} \times^B \cdots \times^B P_{\alpha_{i_m}}/B,
\end{align*}
which is a Bott-Samelson variety. Here, $\widehat{w} = \tilde{w}$ and $\widehat{X}(\widehat{w}) \simeq Z(\tilde{w})$.

For $1 \leq j \leq m$, let $Z({\tilde{w}}(j))$ be the Bott-Samelson variety associated to the subexpression $\tilde{w}(j) = (s_{i_1}, \dots, s_{i_j})$. Consider the morphisms
\[
\widehat{\pi}_j \colon Z({\tilde{w}}(j)) \to G/B, \quad [p_1, \dots, p_j] \mapsto p_1 \cdots p_j B,
\]
and recall the map $\widehat{f}_j$ from Section \ref{GBSV}. For $1 \leq j \leq m$, define the line bundle on $Z(\tilde{w})$ by
\[
\mathcal{L}_j := \widehat{f}_j^* \widehat{\pi}_j^* \mathcal{L}_{G/B}({\omega}_{i_j}),
\]
where $\mathcal{L}_{G/B}({\omega}_{i_j})$ is the line bundle on $G/B$ associated to ${\omega}_{i_j}$. For $\mathbf{n} = (n_1, \dots, n_m) \in \mathbb{Z}^m$, set
\[
\mathcal{L}_{\mathbf{i}, \mathbf{n}} := \mathcal{L}_1^{n_1} \otimes \cdots \otimes \mathcal{L}_m^{n_m}.
\]

\begin{theorem}{\cite[Theorem~3.1]{LT}}\label{thm:LT line bundles}
\begin{enumerate}\ 
    \item The isomorphism classes of $\mathcal{L}_1, \dots, \mathcal{L}_m$ form a basis of $\mathrm{Pic}(Z(\tilde{w}))$. In particular, the map $\mathbb{Z}^m \to \mathrm{Pic}(Z(\tilde{w}))$, $\mathbf{n} \mapsto \mathcal{L}_{\mathbf{i}, \mathbf{n}}$, is an isomorphism.
    \item The line bundle $\mathcal{L}_{\mathbf{i}, \mathbf{n}}$ is (very) ample if and only if $n_j > 0$ for all $j$.
\end{enumerate}
\end{theorem}

\subsection{Line bundles on \texorpdfstring{$\widehat{X}(\widehat{w})$}{X-hat(w-hat)}}

We now construct the line bundles over smooth generalized Bott-Samelson variety $\widehat{X}(\widehat{w})$ and compute the Pic($\widehat{X}(\widehat{w})$). We also give a criterion for a line bundle to be very ample over $\widehat{X}(\widehat{w})$.
In this direction, we first prove the following:
\begin{proposition} \label{line}
We keep the notion as above. The variety $\widehat{X}(\widehat{w})$ is embedded as a closed subvariety inside $\prod_{j=1}^{m} G/P^{w_j}$. The embedding is given by  
\begin{align*} 
\Phi_{\widehat{w}}: \widehat{X}(\widehat{w}) &\to \prod_{j=1}^{m} G/P^{w_j} \\ 
[p_1,\dots,p_m] &\mapsto (p_1 P^{w_1}, p_1 p_2 P^{w_2}, \dots, (\prod_{j=1}^{m} p_j) P^{w_m}),
\end{align*} 
where $p_j \in \overline{(P_{w_j} \cap G_{w_j}) w_j (P^{w_j} \cap G_{w_j})} \subset G_{w_j}$ for all $1 \leq j \leq m$.
\end{proposition} 

\begin{proof}
We proceed by induction on $m$. For $m=1$, it is clear that 
\[
\widehat{X}(\widehat{w}) \simeq X(w_1) \subset G/P^{w_1}.
\] 
We now consider the case $m=2$. We construct the embedding in two steps.

\textbf{Step~1.} Define a morphism 
\begin{align*}
\psi_2: \widehat{X}(\widehat{w}) &\to G/P^{w_1} \times G_{w_2} / (P^{w_2} \cap G_{w_2}) \\
[p_1,p_2] &\mapsto (p_1 P^{w_1}, p_1 p_2 (P^{w_2} \cap G_{w_2})).
\end{align*}

The morphism $\psi_2$ is well-defined: indeed, if $[p_1, p_2] = [q_1, q_2]$ then by the definition of $\widehat{X}(\widehat{w})$, there exist elements $h_i \in P^{w_i} \cap G_{w_i} \subset P^{w_i}$ for $i=1,2$ such that $q_1 = p_1 h_1$ and $q_2 = h_1^{-1} p_2 h_2$. Then it follows 
\[
(p_1 P^{w_1}, p_1 p_2 (P^{w_2} \cap G_{w_2})) = (q_1 P^{w_1}, q_1 q_2 (P^{w_2} \cap G_{w_2})).
\]
We now prove that the morphism $\psi_2$ is injective. If 
\begin{align*}
\psi_2([p_1,p_2]) &= \psi_2([q_1,q_2]) \\
(p_1 P^{w_1}, p_1 p_2 (P^{w_2} \cap G_{w_2})) &= (q_1 P^{w_1}, q_1 q_2 (P^{w_2} \cap G_{w_2})),
\end{align*}
where $p_j, q_j \in \overline{(P_{w_j} \cap G_{w_j}) w_j (P^{w_j} \cap G_{w_j})} \subset G_{w_j}$ for all $1 \leq j \leq 2$. Then, we have
\[
p_1^{-1} q_1 \in P^{w_1}\cap G_{w_1} \quad \text{and} \quad (p_1 p_2)^{-1} (q_1 q_2) \in P^{w_2} \cap G_{w_2}.
\] 
Therefore, there exist $g_1 \in P^{w_1} \cap G_{w_1}$ and $g_2 \in P^{w_2} \cap G_{w_2}$ such that \[p_1^{-1} q_1 = g_1\quad \text{and} \quad (p_1 p_2)^{-1} (q_1 q_2) = g_2.\] Since $q_1 = p_1 g_1$ and $q_2 = g_1^{-1} p_2 g_2$, we conclude that $[p_1, p_2] = [q_1, q_2].$

\textbf{Step~2.} There is also a natural injective morphism:
\begin{align*}
\eta_2: G/P^{w_1} \times G_{w_2} / (P^{w_2} \cap G_{w_2}) &\to G/P^{w_1} \times G/P^{w_2} \\
(g P^{w_1}, h (P^{w_2} \cap G_{w_2})) &\mapsto (g P^{w_1}, h P^{w_2}).
\end{align*}

Now set $\phi_{\widehat{w}} := \eta_2 \circ \psi_2$, which is the desired morphism. Moreover, since the image of a projective variety under a regular morphism is closed, thus the assertion follows for the case $m=2$.

Assume that the statement holds for all $2 \leq m-1$ and we will prove it for $m$. Define 
\begin{align*}
\psi_m: \widehat{X}(\widehat{w}) &\to \prod_{i=1}^{m-1} G/P^{w_i} \times G_{w_m} / (P^{w_m} \cap G_{w_m}) \\
[p_1, \dots, p_m] &\mapsto (p_1 P^{w_1}, \dots, (\prod_{i=1}^{m-1} p_i) P^{w_{m-1}}, (\prod_{i=1}^{m} p_i)(P^{w_m} \cap G_{w_m})),
\end{align*}
where $p_j \in \overline{(P_{w_j} \cap G_{w_j}) w_j (P^{w_j} \cap G_{w_j})} \subset G_{w_j}$ for all $1 \leq j \leq m$. One can easily check that $\psi_m$ is a well-defined morphism by similar arguments as in the case $m=2$. 
We now claim that the morphism $\psi_m$ is injective. Assume that 
\[
\psi_m([p_1, \dots, p_m]) = \psi_m([q_1, \dots, q_m]),
\]
then we get 
\[
(\prod_{i=1}^{l} p_i) P^{w_i} = (\prod_{i=1}^{l} q_i) P^{w_i} \quad \text{for } 1 \leq l \leq m-1,
\]
and 
\[
(\prod_{i=1}^{m} p_i) (P^{w_m} \cap G_{w_m}) = (\prod_{i=1}^{m} q_i) (P^{w_m} \cap G_{w_m}).
\]
This implies 
\[
(\prod_{i=1}^{m} p_i)^{-1} (\prod_{i=1}^{m} q_i) \in P^{w_m} \cap G_{w_m}.
\] 
Let $g_m := (\prod_{i=1}^{m} p_i)^{-1} (\prod_{i=1}^{m} q_i)$. By induction, there exist $g_j \in P^{w_j} \cap G_{w_j}$ such that 
\[
g_j = (\prod_{i=1}^{j} p_i)^{-1} (\prod_{i=1}^{j} q_i) \quad \text{for } 1 \leq j \leq m-1.
\]
Then we have 
\[
q_j = g_{j-1}^{-1} p_j g_j \quad \text{for } 2 \leq j \leq m, \quad \text{and} \quad q_1 = p_1 g_1.
\]
This finally leads us to $[p_1, \dots, p_m] = [q_1, \dots, q_m].$ 
This proves the claim. Hence by using the natural injective morphism
\[
\eta_m:\prod_{i=1}^{m-1} G/P^{w_i} \times G_{w_m} / (P^{w_m} \cap G_{w_m}) \hookrightarrow \prod_{i=1}^{m} G/P^{w_i},
\] the composition morphism $\Phi_{\widehat{w}} := \eta_m \circ \psi_m$ gives the desired result.
\end{proof}
\begin{remark}
    In Step $1$ of the proof of Proposition \ref{line}, in general it does not work if we take the co-domain of the morphsim $\psi_{2}$ to be \[
    Z:=G_{w_1}/(P^{w_1}\cap G_{w_1})\times G_{w_2}/(P^{w_2}\cap G_{w_2}).
    \]As in $SL_{3}(\mathbb{C})$, let $\widehat{w}=(s_1,s_2)$ be the admissible generalized reduced decomposition of $w=s_1s_2\in W$. Note that for a simple root $\alpha$, $G_{s_{\alpha}}=L_{s_{\alpha}}$, where $L_{s_{\alpha}}$ is the Levi factor of the parabolic subgroup $P_{s_{\alpha}}$ of $SL_{3}(\mathbb{C})$. Moreover, $P^{s_{\alpha}}\cap G_{s_{\alpha}}$ is a Borel subgroup of $L_{s_{\alpha}}$. Then \[
    Z\simeq \mathbb{P}^1\times \mathbb{P}^1.
    \] But by Subsection \ref{subsec:line},  the domain of $\psi_2$ is isomorphic to \[
   Z(s_1s_2) =P_{\alpha_1}\times^B P_{\alpha_2}/B.
    \] These together imply $\psi_2$ is an isomorphism. That means $Z(s_1s_2)\simeq \mathbb{P}^1\times \mathbb{P}^1$ which is a contradiction.
\end{remark}

To describe line bundles on $\widehat{X}(\widehat{w})$, we use an equivalent definition for the Bott-Samelson variety $Z({\tilde{w}})$ associated with a reduced expression of $w\in W$ constructed by Magyar in \cite{Mag98}.

Let $\tilde{w} = (s_{i_1}, \dots, s_{i_m})$ be an admissible generalized reduced expression of $w=s_{i_1} \cdots s_{i_m}\in W$ and let $\pi_{j}: G/B \to G/P_{\alpha_{i_{j}}}$ be the projection morphism whose fibers are isomorphic to $P_{\alpha_{i_j}}/B$. Note that $P_{\alpha_{i_j}}$ is the minimal parabolic subgroup, so $P_{\alpha_{i_j}}/B \cong \mathbb{P}^1$. Let $\mathbb{P}(x,\alpha_{i_j})$ denotes the projective line $\pi^{-1}_{j}(\pi_{j}(x))$ for $x\in G/B$.

For any $x\in G/B$ the map $G/B\to G/P^{\alpha_{i_j}}$ restricted to $\mathbb{P}(x,\alpha_{i_j})$ is an isomorphism onto its image which is called $\overline{\mathbb{P}}(x,\alpha_{i_j})$. If $y\in G/B$ and $\overline{x}\in \overline{\mathbb{P}}(y,\alpha_{i_{j}})$ for $x\in \mathbb{P}(y, \alpha_{i_{j}})$ we abuse notation by writing $\overline{\mathbb{P}}(\overline{x}, \alpha_{i_{j+1}})=\overline{\mathbb{P}}({x}, \alpha_{i_{j+1}})$. Under the morphism $\Phi_{\tilde{w}}$ defined in Proposition \ref{line}, the variety $Z(\tilde{w})$ is isomorphic to \[
 \tilde{X}(\tilde{w}):= \Big\{\, (x_1,\dots,x_m) \in \prod_{j=1}^{m} G/P^{s_{i_j}} \;\Big|\; 
x_0 = 1 \ \text{and} \ x_j \in \overline{\mathbb{P}}(x_{j-1}, \alpha_{i_j}) \ \text{for all } 1 \leq j \leq m \,\Big\}
\]
 see, \cite[Theorem~1]{Mag98} and \cite[Section~2]{Per07}. In this setting, for $1\leq j \leq m$, we have natural projections $pr_j: \tilde{X}(\tilde{w})\to G/P^{s_{i_j}}$ and we define the line bundles \[\mathcal{M}_j:= pr_j^*(\mathcal{O}_{G/P^{s_{i_j}}}(\omega_{\alpha_{i_j}}))\] on $\tilde{X}(\tilde{w})$.
 Recall the maps $\widehat{f}_j$ and $\widehat{\pi}_j$ described in Section \ref{GBSV} and Section \ref{sec:line}, respectively. Let $\pi^j:G/B\to G/P^{s_{i_j}}$ be a natural morphism.
 
 \begin{remark}\label{remark1}
Consider the following commutative diagram
 \begin{center}
\begin{tikzpicture}[scale=1.5, >=Stealth, every node/.style={font=\small}]

  \node (P1) at (-1,6) {$Z(\tilde{w})$};
  \node (P2) at (-1,4.5) {$Z(\tilde{w}(j))$};
  \node (P3) at (2.5,4.5) {$G/B$};
  \node (P4) at (2.5,6) {$\tilde{X}(\tilde{w})\subset \prod_{k=1}^{m}G/P^{s_{i_k}}$};
  \node (P5) at (4,4.5) {$G/P^{s_{i_j}}$};

  \draw[->] (P1) -- node[left] {$\widehat{f}_j$} (P2);
  \draw[->] (P2) -- node[below] {$\widehat{\pi}_j$}(P3);
  \draw[->] (P1) -- node[above] {$\Phi_{\tilde{w}}$}(P4);
  \draw[->] (P4) -- node[right=8pt] {$pr_j$}(P5);
  \draw[->] (P3) -- node[below] {$\pi^j$}(P5);
\end{tikzpicture}
\end{center}

 Then, we have 
 \begin{center}
$pr_j \circ \Phi_{\tilde{w}}= \pi^j\circ\widehat{\pi}_{j}\circ \widehat{f}_j $
 \end{center}
 for all $1\leq j\leq m$. Therefore, the pull-back of line bundles $ \mathcal{M}_j$ under $\Phi_{\tilde{w}}$ to $Z({\tilde{w}})$ are same as line bundles $\mathcal{L}_j$ defined in the beginning of Section \ref{sec:line}, that is,\[
\mathcal{L}_j = \Phi_{\tilde{w}}^*\big( \mathcal{M}_j \big)
\quad \text{for all } 1 \leq j \leq m.
\]
 \end{remark}
 
\begin{proposition}

The set $ \big\{ \mathcal{M}_j \mid 1 \leq j \leq m \big\} $ forms a basis for $ \operatorname{Pic}(\tilde{X}(\tilde{w})) $.
\end{proposition}
\begin{proof}
    Since $Z(\tilde{w})\simeq \tilde{X}(\tilde{w})$ and by Theorem \ref{thm:LT line bundles}(1), the set $ \big\{ \mathcal{L}_j \mid 1 \leq j \leq m \big\} $ forms a basis for $Z(\tilde{w})$. Thus, by Remark \ref{remark1} we get the result.
\end{proof}
We now describe a basis for the Chow group $A_1(\tilde{X}(\tilde{w}))$ which is dual to the basis $ \big\{ \mathcal{M}_j \mid 1 \leq j \leq m \big\} $ for $ \operatorname{Pic}(\tilde{X}(\tilde{w}))$.
For this we introduce $B$-stable (under the diagonal action) curves $Y_{\alpha_{i_j}}$ inside $\tilde{X}(\tilde{w})$, for $1\leq j \leq m$, defined as \begin{equation}\label{eq:curves}
Y_{\alpha_{i_j}} := \Big\{\, (x_1, \dots, x_m) \in \prod_{k=1}^{m} G/P^{s_{\alpha_k}} \;\Big|\; 
x_k = 1 \text{ for } k \neq j,\ \text{and } x_j \in \overline{\mathbb{P}}(1, \alpha_{i_j}) \,\Big\},
\end{equation}
so that $Y_{\alpha_{i_j}}$ is isomorphic to $\overline{\mathbb{P}}(1, \alpha_{i_j})$, see \cite[Lemma 2.12]{Per07}. We also note that
\[Y_{\alpha_{i_j}} \xrightarrow[pr_{j}]{\simeq} \overline{\mathbb{P}}(1,\alpha_{i_j})= P_{\alpha_{{i_j}}}/P^{s_{i_j}}= X(s_{i_j})\subset G/P^{s_{i_j}}\]
and for $k\neq j$
\[
Y_{\alpha_{i_j}} \xrightarrow[pr_{k}]{} eP^{s_{{i_k}}}/P^{s_{i_k}}.
\]
Thus, by using projection formula, we have 
\begin{align*}
     \mathcal{M}_j\cdot[Y_{\alpha_{i_k}}]&= {pr_j}_{*}\big(\mathcal{M}_j\cdot[Y_{\alpha_{i_k}}]\big) \\ &=\mathcal{O}_{G/P^{s_{i_{j}}}}(\omega_{\alpha_{i_j}})\cdot {pr_j}_{*}[Y_{\alpha_{i_k}}] \\ &=\begin{cases}
     \langle \omega_{\alpha_{i_j}}, \alpha_{i_k}^\vee \rangle= 1 &\text{if}~ j=k\\0 & \text{otherwise}.\end{cases}\\
\end{align*}
This together with \cite[Proposition 2.14]{Per07} gives the following:

\begin{proposition}\label{prop:dual}
The families $ \big\{ \mathcal{M}_j \mid 1 \leq j \leq m \big\} $ and $ \big\{ [Y_{\alpha_{i_j}}] \mid 1 \leq j \leq m \big\} $ are dual to each other. Moreover, the family $ \big\{ [Y_{\alpha_{i_j}}] \mid 1 \leq j \leq m \big\} $ form a basis of $A_1(\tilde{X}(\tilde{w}))$.
\end{proposition}

Since $\Phi_{\tilde{w}}$ is a $B$-equivariant isomorphism between $Z(\tilde{w})$ and $\tilde{X}(\tilde{w})$, so the pull-back $\Phi_{\tilde w}^*\big(Y_{\alpha_{i_j}} \big)$ of any curve $Y_{\alpha_{i_j}}$ is again a $B$-stable curve inside $Z(\tilde{w})$.

In general, the push-forward of a curve need not be again a curve. But in our case, we construct an indexing set depends on an expression of $w\in W$ such that the push-forward under the $B$-equivariant morphism $\pi$ (see Equation \ref{BS resoltion}) of curves corresponding to the indexing set are again curves.

Let $w=s_{i_1}\cdots s_{i_m} $ be a reduced (admissible generalized) decomposition of $w\in W$, let  
\begin{equation}\label{eq:beta} \beta_1:=\alpha_{i_1},\beta_2:= s_{i_1}(\alpha_{i_2}), \dots, \beta_m:= s_{i_1}\cdots s_{i_{m-1}}(\alpha_{i_{m}}).\end{equation}

Note that
\[
R^+ \cap w(R^-)= \{\beta_1, \beta_2, \dots, \beta_m\}.
\]
Recall that $R^+$ and $R^-$ denote the sets of positive and negative roots corresponding to $B$, respectively.

The following lemma is useful in the computations of the subsequent results.

\begin{lemma}\label{lem: Supp(w)}
    For any $w\in W$, we have $ S\setminus I^{w}\subseteq \operatorname{Supp}(w)$.
\end{lemma}

\begin{proof}
   Let 
$
w = s_{i_1}s_{i_2}\cdots s_{i_m}
$
be a reduced expression of an element \( w \in W \). 
Then, for any simple root \( \alpha \in S \), one has
\begin{equation}\label{eq:action-on-beta}
    s_{i_1}s_{i_2}\cdots s_{i_m}(\alpha)
    = \alpha 
      - \sum_{j = 1}^{m} 
        \langle \alpha, \alpha_{i_j}^{\vee} \rangle \, \beta_j.
\end{equation}
Moreover, if $\alpha\in S\setminus \text{Supp}(w)$ then $\langle \alpha, \alpha_{i_j}^{\vee} \rangle \leq 0,$
for all $1\leq j \leq m$. Then by Equation \ref{eq:action-on-beta}, we have $w(\alpha) = s_{i_1}s_{i_2}\cdots s_{i_m}(\alpha)  >0.$ Therefore, $ S\setminus \text{Supp}(w)\subseteq I^w$ and $ S\setminus I^w \subseteq \text{Supp}(w)$.  
\end{proof}

Under the same notations as above for any $\alpha_{i_j}\in S\setminus {I}^w \subseteq \text{Supp}(w)$.  We define
\begin{equation}\label{eq:m(j)}
    m(j):= \text{max}\{q \mid s_{i_q}=s_{i_j} \} \quad \text{and}\quad A_w:= \{\alpha_{i_{m(j)}} \mid \alpha_{i_j}\in S\setminus {I}^w \}.
\end{equation}

Note that $A_w= S\setminus I^w$.

\begin{example}\label{ex:index}
In $SL_4(\mathbb{C})$ consider the maximal element 
\[w=s_1s_2s_3s_1s_2s_1 = s_{i_1}s_{i_2}s_{i_3}s_{i_4}s_{i_5}s_{i_6}.\] Here 
\[S\setminus I^w = \text{Supp}(w)= \{\alpha_1= \alpha_{i_1}=\alpha_{i_4}=\alpha_{i_6}, \alpha_2=\alpha_{i_2}=\alpha_{i_5},\alpha_3=\alpha_{i_3}\}.\] Then by Equation \ref{eq:m(j)}, $m(1)=m(4)=m(6)=6$, $m(2)=m(5)=5$, $m(3)=3$ and \[A_w=\{\alpha_1 = \alpha_{i_6}, \alpha_2 = \alpha_{i_5}, \alpha_3= \alpha_{i_3} \}.\]
\end{example}

\begin{remark}
  In Example \ref{ex:index}, the choice for the index of a simple root inside $S\setminus I^w$ is more than one. For $\alpha_1 \in S\setminus I^w$ we have three choices $i_1,i_4,i_6$ for the index. But to construct $A_w$ from it we choose the maximum among them, and as a set $A_w$ and $S\setminus I^w$ are the same.
\end{remark}

\begin{lemma}\label{contracted}
For any $1\leq j \leq m$ such that $\alpha_{i_j}\in S\setminus I^w$. The push-forward $\pi_{*}\big(\Phi_{\tilde{w}}^*\big(Y_{\alpha_{i_{m(j)}}}\big)\big)$ of a $B$-stable curve $\Phi_{\tilde{w}}^*\big(Y_{\alpha_{i_{m(j)}}}\big)$ is again a $B$-stable curve.
\end{lemma}
\begin{proof}

  For any $\alpha_{i_j}\in S \setminus I^w \subseteq \text{Supp}(w)$, we have the following commutative diagram 
\begin{center}
\begin{tikzpicture}[scale=1.5, >=Stealth, every node/.style={font=\small}]

  \node (P1) at (-1,6) {$Z(\tilde{w})$};
  \node (P2) at (-1,4.5) {$X(w)\subset G/P^w$};
  \node (P3) at (2.5,6) {$\tilde{X}(\tilde{w})\subset \prod_{k=1}^{m}G/P^{s_{i_k}}$};
  \node (P4) at (2.5,4.5) {$G/P^{s_{i_{m(j)}}}$};

  \draw[->] (P1) -- node[left] {$\pi$} (P2);
  \draw[->] (P2) -- node[below] {$\pi^{m(j)}$}(P4);
  \draw[->] (P1) -- node[above] {$\Phi_{\tilde{w}}$}(P3);
  \draw[->] (P3) -- node[right=8pt] {$pr_{m(j)}$}(P4);
\end{tikzpicture}
\end{center}
where $pr_{m(j)}$ is the projection and $\pi^{m(j)}:G/P^w\to G/P^{s_{i_{m(j)}}}$ is the natural morphism.
Note that for any curve $Y_{\alpha_{i_{m(j)}}}$, we have 
\begin{align*}
\pi^{m(j)}\circ\pi\big(\Phi_{\tilde{w}}^*\big(Y_{\alpha_{i_{m(j)}}}\big) \big)&= pr_{m(j)}\circ \Phi_{\tilde{w}}\big(\Phi_{\tilde{w}}^*\big(Y_{\alpha_{i_{m(j)}}}\big)\big)\\
    &=pr_{m(j)}\big(Y_{\alpha_{i_{m(j)}}}\big)\\
    &= \overline{\mathbb{P}}(1,\alpha_{i_{m(j)}}).
\end{align*}
Also, all morphisms used in commutative diagram are $B$-equivariant. Therefore, $\pi\big(\Phi_{\tilde{w}}^*\big(Y_{\alpha_{i_{m(j)}}}\big) \big)$ should again be a $B$-stable curve inside $X(w)$ otherwise the above equality does not hold.
\end{proof}

Recall the morphism $\Phi_{\widehat{w}}$ defined in Proposition \ref{line} and let $pr_i : \prod_{j=1}^{m} G/P^{w_j} \to G/P^{w_i}$ be the projections, and set
\[
\mathcal{L}_{i,\alpha} := \Phi_{\widehat{w}}^* pr_i^* \mathcal{O}_{G/P^{w_i}}(\omega_\alpha),
\]
where $\mathcal{O}_{G/P^{w_i}}(\omega_{\alpha})$ is the homogenous line bundle on $G/P^{w_i}$ associated to the fundamental weight $\omega_{\alpha}$ with $\alpha\in (S\setminus I^{w_i})\cap \text{Supp}(w_i)$, for $1\leq i \leq m$.

By Proposition \ref{line}, we have the restriction map:
\begin{align*}
       \text{res}: \text{Pic}(\prod_{i=1}^{m}G/P^{w_i}) &\to \text{Pic}(\widehat{X}(\widehat{w})).
       \end{align*}

\begin{theorem}\label{thm:linebundles} We have:
\begin{enumerate} 
    \item The restriction map $\it res$ is an isomorpshim.  
    \item  The Picard group of $\widehat{X}(\widehat{w})$ is isomorphic to $\bigoplus_{i=1}^{m}\mathbb Z^{\#(S\setminus I^{w_i})}$. 
    \item  Any line bundle $\mathcal L$ on $\widehat{X}(\widehat{w})$ is of the form \begin{equation}\label{eq:linebundles}
\mathcal{L} = \bigotimes_{i=1}^m \bigotimes_{\alpha \in S \setminus I^{w_i}} \mathcal{L}_{i, \alpha}^{\otimes a_{i,\alpha}}, \quad a_{i,\alpha}\in \mathbb Z.
\end{equation}
\end{enumerate} 
 \end{theorem}
\begin{proof}
    We use induction to prove the statement. For $m=1$, 
    \[
    \widehat{X}(\widehat{w})\simeq X(w_1)\subset G/P^{w_1}
    \]
    and since for any $w\in W$, we have \[\text{Pic}(X(w))  \simeq\bigoplus_{\alpha\in (S\setminus I^{w})\cap \text{Supp}(w)}\mathbb{Z} \omega_{\alpha}\] and by Lemma \ref{lem: Supp(w)}, $S\setminus I^{w}\subseteq \text{Supp}(w)$, then 
    \[
    \text{Pic}(X(w_1))  \simeq\bigoplus_{\alpha\in S\setminus I^{w_1}}\mathbb{Z} \omega_{\alpha}\simeq  \mathbb{Z}^{\#(S \setminus I^{w_1})}
    \]
            with this together with Lemma \ref{contracted} and Proposition \ref{prop:dual}, the set $\{\mathcal{L}_{1,\alpha}\mid \alpha\in S\setminus I^{w_1}\}$ forms a basis for $\operatorname{Pic}X(w_1)$. So, we are done for $m=1.$ For the case $m=2$, recall $\widehat{X}(\widehat{w})$ is the locally trivial fibration over $\widehat{X}(w_1)$ with fibers are isomorphic to $\widehat{X}(w_2).$ For $i=1,2$, the variety $\widehat{X}({w}_i)$ is smooth by Remark \ref{rmk:smoothness} and also note that these Schubert varieties are rational varieties, and the Pic$(\widehat{X}(w_i))$ is projective $\mathbb{Z}$-module. Then by \cite[Proposition~2]{BB20}, 
            \[\text{Pic}(\widehat{X}(w)) \simeq \text{Pic}(\widehat{X}(w_1))\oplus \text{Pic}(\widehat{X}(w_2)) \simeq \bigoplus_{i=1}^{2}\mathbb Z^{\#(S\setminus I^{w_i})}.\]
            Now assume that this result holds for $2\leq m-1$, we will prove it for $m$. Since the variety $\widehat{X}(\widehat{w})$ is the locally trivial fribration over $\widehat{X}(w_1)$ with fiber $\widehat{X}(w_2,\dots,w_m)$. Since $\widehat{X}(w_2,\dots,w_m)$ is smooth and rational variety as being birational to the Schubert variety $X(w_2\cdots w_m)$, and by induction $\text{Pic}(\widehat{X}(w_2,\dots,w_m))\simeq \bigoplus_{i=2}^{m}\mathbb{Z}^{\#(S\setminus I^{w_i})}$ is a projective $\mathbb{Z}$-module. Therefore, \begin{align*}
   \text{Pic}(\widehat{X}(\widehat{w})) &\simeq \text{Pic}(\widehat{X}(w_1))\oplus \text{Pic}(\widehat{X}(w_2,\dots,w_m))\\
   &\simeq \bigoplus_{i=1}^{m}\mathbb Z^{\#(S\setminus I^{w_i})}.
   \end{align*} 
   Thus, the $\text{Pic}(\widehat{X}(\widehat{w}))$ is a free abelian group of rank $\sum_{i=1}^{m}\#(S \setminus I^{w_i})$ and by the construction of $\mathcal{L}_{i,\alpha}$, the set $\{ \mathcal{L}_{i,\alpha} \mid 1 \leq i \leq m,\ \alpha \in S \setminus I^{w_i} \}$ forms a basis for $\text{Pic}(\widehat{X}(\widehat{w})).$ 
Then by Proposition \ref{line}, the map
\begin{align*}
       \text{res}: \text{Pic}(\prod_{i=1}^{m}G/P^{w_i}) &\to \text{Pic}(\widehat{X}(\widehat{w}))
       \end{align*} is surjective. By again using \cite[Proposition~2]{BB20}, we have 
 \[
 \text{Pic}(\prod_{i=1}^{m}G/P^{w_i})\simeq \bigoplus_{i=1}^{m}\text{Pic}(G/P^{w_i})\simeq\bigoplus_{i=1}^{m}\mathbb Z^{\#(S\setminus I^{w_i})}.
 \] 
 This yields the res map is an isomorphism. Thus, the proof follows.
\end{proof}
The line bundles $\mathcal{L}$ over $\widehat{X}(\widehat{w})$ obtained in Equation \ref{eq:linebundles} of Theorem \ref{thm:linebundles} are also referred as line bundles associated with $a_{i,\alpha}$'s.

Now we construct indexing sets for a generalized reduced decomposition of an element $w\in W$ as same as we constructed in Equation \ref{eq:m(j)}. Recalling Lemma~\ref{lem:reduced decomposition}, let $\widehat{w} = (w_1, \dots, w_m)$ and \[\tilde{w} = (s_{1,1}, \dots, s_{r_1,1}, \dots, s_{1,m}, \dots, s_{r_m,m})\] be two generalized admissible reduced decompositions of $w \in W$. 
We denote the expression $\tilde{w}$ by the sequence 
\[{\bf i}=(i_{1,1}, \ldots, i_{r_1, 1}, \ldots, i_{1, r_m,},\ldots, i_{r_m, m}).\]
For $1\leq j \leq m$ and $1\leq d\leq r_j$, we set \[
m_j(d): = \max \{q \mid \ s_{q,j}= s_{d,j}
\}.
\]
\begin{remark}\label{remark2}
 By the definition of $m_j(d)$ and $S\setminus I^{w_j}\subseteq \text{Supp}(w_j)$ (by Lemma \ref{lem: Supp(w)}), then we have  \[
 A_{w_j}:= \{ \alpha_{m_j{(d)},j} \mid \alpha_{d,j}\in S\setminus I^{w_j} \}=S\setminus I^{w_j}.
 \]
\end{remark}
\smallskip

\begin{example}\label{ex:indexing sets}
    In $SL_4(\mathbb{C})$, let \[\widehat{w} = (w_1, w_2) = (s_1s_2s_3s_1, s_2)\] and \[\tilde{w} = (s_1, s_2, s_3, s_1, s_2) = (s_{1,1}, s_{2,1}, s_{3,1}, s_{4,1}, s_{1,2})\] be two admissible generalized reduced decompositions of $w = s_1s_2s_3s_1s_2$. Then
  \[
    S\setminus I^{w_1}= \{\alpha_1=\alpha_{1,1}=\alpha_{4,1}, \alpha_3=\alpha_{3,1}\} \quad \text{and} \quad S\setminus I^{w_2}= \{\alpha_2=\alpha_{1,2}\}.
    \]
    Also, \[
    m_1(1)=m_1(4)=4 \quad \text{and} \quad m_2(1)=1.
    \]
    By using Remark \ref{remark2}, we have
    \[
A_{w_1} = \{ \alpha_{4,1} = \alpha_1, \alpha_{3,1} = \alpha_3 \} = S \setminus I^{w_1}
\quad \text{and} \quad
A_{w_2} = \{ \alpha_{1,2} = \alpha_2 \} = S \setminus I^{w_2}.
\]
\end{example}

\begin{proposition}\label{pull back} 
    Let $\mathcal{L}$ be a line bundle on $\widehat{X}(\widehat{w})$ associated with integers $a_{j,\alpha} $ for $1 \leq j \leq m$ and $\alpha:=\alpha_{d,j} \in S \setminus I^{w_j}$ for some $1\leq d\leq r_j$. Let 
    \[
    \tilde{\pi}: Z(\tilde{w}) \to \widehat{X}(\widehat{w})
    \]
    be the natural birational map (see Section~\ref{sec:linktoBS}). Then, the pull-back $\tilde{\pi}^*\mathcal{L} = \mathcal{L}_{\mathbf{i}, \mathbf{n}}$, where $\mathbf{n}=(n_{1,1},\dots,n_{r_1,1},\dots,n_{1,m},\dots,n_{r_m,m})$ and the integers $n_{q,t}$ for $1\leq t\leq m$, $1\leq q \leq r_t$ with  \[
n_{q,t} = 
\begin{cases}
    a_{j,\alpha} & \text{if } t=j\,\, \text{and}\,\, q= m_j(d) , \\
    0 & \text{otherwise}.
\end{cases}
\]
\end{proposition}

\begin{proof}
    Since $S\setminus I^{w_j}= A_{w_j}$, we have $\alpha= \alpha_{d,j}=\alpha_{m_j(d),j}$. Since for any $1\leq j\leq m$ and $1\leq d\leq r_j$, we have a natural morphism $\pi^{m_j(d),j}: G/P^{w_j}\to G/P^{s_{m_j(d),j}}$ and we also have the following commutative diagram:
   \begin{center}
\begin{tikzpicture}[scale=1.5, >=Stealth, every node/.style={font=\small}]

  \node (P1) at (-1,6) {$Z(\tilde{w})$};
  \node (P2) at (-1,4.5) {$\widehat{X}(\widehat{w})$};
  \node (P3) at (1.8,4.5) {$\prod_{k=1}^{m}G/P^{w_k}$};
  \node (P4) at (2.5,6) {$\tilde{X}(\tilde{w})\subset \prod_{k=1}^{m}\prod_{i=1}^{r_k} G/P^{s_{i,k}}$};
  \node (P5) at (6,4.5) {$G/P^{s_{m_j(d),j}}$};
  \node (P6) at (4,4.5) {$G/P^{w_j}$};

  \draw[->] (P1) -- node[left] {$\tilde{\pi}$} (P2);
  \draw[->] (P2) -- node[below] {$\Phi_{\widehat{w}}$}(P3);
  \draw[->] (P1) -- node[above] {$\Phi_{\tilde{w}}$}(P4);
  \draw[->] (P4) -- node[right=8pt] {$pr_{m_j(d),j}$}(P5);
  \draw[->] (P6) -- node[below] {$\pi^{m_j(d),j}$}(P5);
  \draw[->] (P3) -- node[below] {$pr_j$}(P6);
\end{tikzpicture}
 \end{center}
   where $pr_{i,j}:\prod_{1\leq p\leq m}\prod_{1\leq q\leq r_p}G/P^{s_{q,p}}\to G/P^{s_{i,j}}$ is the projection. Then, we have \begin{equation}\label{eq:equality of maps}
   pr_{m_j(d),j}\circ \Phi_{\tilde{w}}= \pi^{m_j(d),j}\circ pr_j \circ \Phi_{\widehat{w}}\circ \tilde{\pi}.
   \end{equation}
   Therefore, by Remark \ref{remark1} and the above equality of maps, we obtain:\[
   \mathcal{L}_{m_j(d),j}= \tilde{\pi}^*\mathcal{L}_{j,\alpha},
   \] 
    Hence, by the definition of $\mathcal L_{{\bf i}, n}$, the result follows.
\end{proof}
\begin{example}\label{eg1} Keep the notations same as discussed in Example \ref{ex:indexing sets}. 
If $\mathcal{L}$ is a line bundle over $\widehat{X}(\widehat{w})$ corresponding to $(a_{1,\alpha_1}, a_{1,\alpha_3}, a_{2,\alpha_2}) \in \mathbb{Z}^3$, then the line bundle $\tilde{\pi}^*\mathcal{L}$ over $Z(\tilde{w})$ corresponds to the tuple $(0, 0, a_{1,\alpha_3}, a_{1,\alpha_1}, a_{2,\alpha_2}) \in \mathbb{Z}^5$.
\end{example}

By Remark \ref{remark2}, $A_{w_j}= S\setminus I^{w_j}$ for every $1\leq j \leq m$. Also, recall the curves $Y_{\alpha_{i_j}}$ defined in Equation \ref{eq:curves}. In our setting, for $1\leq j\leq m$ and $1\leq d\leq r_j$ such that $\alpha_{d,j}\in S\setminus I^{w_j}$, we set\[
\widehat{Y}_{m_j(d)}:= \tilde{\pi}_*\big(\Phi^*_{\tilde{w}}(Y_{\alpha_{m_j(d),j}})\big).
\]

\begin{lemma}
    Every member of the family $\{\widehat{Y}_{m_j(d)}\mid 1\leq j\leq m, 1\leq d\leq r_j \}$ is a $B$-stable curve inside $\widehat{X}(\widehat{w})$.
\end{lemma}
\begin{proof}
    We apply $ \Phi^*_{\tilde{w}}(Y_{\alpha_{m_j(d),j}})$ on the both sides of Equation \ref{eq:equality of maps}, thus we have\begin{align*} \pi^{m_j(d),j}\circ pr_j \circ \Phi_{\widehat{w}}\circ \tilde{\pi}(\Phi^*_{\tilde{w}}(Y_{\alpha_{m_j(d),j}})\big)   
    &= pr_{m_j(d),j}\circ \Phi_{\tilde{w}}\big(\Phi^*_{\tilde{w}}(Y_{\alpha_{m_j(d),j}})\big)\\\pi^{m_j(d),j}\circ pr_j \circ \Phi_{\widehat{w}}\big(\widehat{Y}_{m_j(d)}\big)&= pr_{m_j(d),j}(Y_{\alpha_{m_j(d),j}})\\
    &= \overline{\mathbb{P}}(1,\alpha_{m_j(d),j}).
   \end{align*} 
   If $\Phi^*_{\tilde{w}}\big(Y_{\alpha_{m_j(d),j}}\big)$ is contracted to a point by $\tilde{\pi}$ then the above equality does not hold. Moreover, all morphisms used here are $B$-equivariant this implies that $\widehat{Y}_{m_j(d)}$ is again a $B$-stable curve inside $\widehat{X}(\widehat{w})$. Hence, the result follows.
\end{proof}

\begin{proposition}\label{dual}
   The families $\{[\widehat{Y}_{m_j(d)}]\mid 1\leq j\leq m, 1\leq d\leq r_j \}$ and $\{\mathcal{L}_{j,\alpha} \mid 1 \leq j \leq m \ \text{and} \ \alpha \in S \setminus I^{w_j} \}$ are dual to each other.
   Moreover, the set $\{[\widehat{Y}_{m_j(d)}]\mid 1\leq j\leq m, 1\leq d\leq r_j \}$ forms a basis of $A_1(\widehat{X}(\widehat{w}))$.
\end{proposition}
\begin{proof}
   Let us fix a $\alpha \in S\setminus I^{w_j}$, then $\alpha$ is of $\alpha_{d,j}$ form for some $1\leq d \leq r_j$. Also by Remark \ref{remark2}, $\alpha=\alpha_{m_j(d),j}$. By the argument in the proof of Proposition \ref{pull back}, $\tilde{\pi}^*(\mathcal{L}_{j,\alpha}) = \mathcal{L}_{m_j(d),j}$. Also fix a $\beta = \alpha_{k,u} \in S\setminus I^{w_u}$, where $1\leq u \leq m$ and $1\leq k\leq r_u $. Then, by the projection formula, we have:
    \begin{align*}
      [\widehat{Y}_{m_u(k)}]\cdot \mathcal{L}_{j,\alpha}  &= \tilde{\pi}_*\Big(\Phi_{\tilde{w}}^*[Y_{\alpha_{m_u(k),u}}]\Big) \cdot \mathcal{L}_{j,\alpha}\\ 
        &= \tilde{\pi}_*\Big(\Phi_{\tilde{w}}^*[Y_{\alpha_{m_u(k),u}}] \cdot \tilde{\pi}^*(\mathcal{L}_{j,\alpha})\Big)\\
        &= \Phi_{\tilde{w}}^*[Y_{\alpha_{m_u(k),u}}] \cdot \mathcal{L}_{m_j(d),j} \\
        &= [Y_{\alpha_{m_u(k),u}}] \cdot \mathcal{M}_{m_j(d),j}, \quad \text{where}~ \mathcal{M}_{m_j(d),j}\in \text{Pic}(\tilde{X}(\tilde{w})),\  \text{see Remark }\ref{remark1} \\
        &= \begin{cases}
            \delta_{m_u(k),m_j(d)} & \text{if } u = j, \\
            0 & \text{if } u\neq j
        \end{cases}\quad (\text{By Proposition} ~\ref{prop:dual}).
    \end{align*} 
    The second assertion follows directly from the first part of the proof and Theorem~\ref{thm:linebundles}.
\end{proof}

\begin{theorem} \label{thm:very ample}
    We keep the notation as in Theorem~\ref{thm:linebundles}. Let $\mathcal{L}$ be a line bundle on $\widehat{X}(\widehat{w})$. Then $\mathcal{L}$ is very ample (nef) if and only if 
    \[
    a_{j,\alpha} > 0 ~(\text{resp.} \geq 0) \quad \text{for all } 1 \leq j \leq m \text{ and for all } \alpha \in S \setminus I^{w_j}.
    \]
\end{theorem}

\begin{proof}
Let 
\[
\mathcal{L} = \bigotimes_{j=1}^{m} \bigotimes_{\alpha \in (S \setminus I^{w_j})} \mathcal{L}^{a_{j,\alpha}}_{j,\alpha}
\]
be a very ample line bundle over $\widehat{X}(\widehat{w})$. Then for any $1\leq j\leq m$, by Proposition~\ref{dual}, we have:
\[
a_{j,\alpha} = \mathcal{L} \cdot [\widehat{Y}_{m_j(d)}] > 0,
\]
where $\alpha = \alpha_{m_j(d),j} \in S \setminus I^{w_j}$. Therefore, by Remark~\ref{remark2}, we conclude that $a_{j,\alpha} > 0$ for all $1 \leq j \leq m$ and for all $\alpha \in S \setminus I^{w_j}$.

\smallskip

Conversely, suppose that $a_{j,\alpha} > 0$ for all $1 \leq j \leq m$ and for all $\alpha \in S \setminus I^{w_j}$. Then, by Theorem~\ref{thm:linebundles}(1), we have:
$
\text{res}(\mathcal{M}) = \mathcal{L},
$
where 
\begin{align*}
\mathcal{M} &= \bigotimes_{j=1}^{m} \bigotimes_{\alpha \in (S \setminus I^{w_j})} pr_j^* \mathcal{O}_{G/P^{w_j}}(a_{j,\alpha} \omega_\alpha)\\
&= \bigotimes_{j=1}^{m}pr_j^* \mathcal{O}_{G/P^{w_j}}\Big(\sum_{\alpha\in (S\setminus I^{w_j})}a_{j,\alpha} \omega_\alpha\Big)
\end{align*}
is a line bundle over $\prod_{j=1}^{m} G/P^{w_j}$. Since $\mathcal{M}$ is very ample (as it corresponds to regular dominant weights), it follows that $\mathcal{L}$ is also very ample. 
The proof of the nefness condition is similar.
\end{proof}

As an immediate consequence, we obtain the following corollary.

\begin{corollary}\label{cor:ample}
   Any ample line bundle over $\widehat{X}(\widehat{w})$ is very ample.
\end{corollary}


\section{Upper bound for the Gromov width of \texorpdfstring{$\widehat{X}(\widehat{w})$}{X-hat(w-hat)}}\label{sec:UB}

\subsection{Families of Rational Curves}In this section, we briefly recall some notions related to rational curves, following \cite[Chapter~II]{Kol96}; see also \cite[Section~2]{BK21} for further background.

Let $X$ be a smooth projective variety. Consider the scheme of morphisms $\text{Hom}(\mathbb{P}^1, X)$, and let $\text{Hom}_{\text{bir}}(\mathbb{P}^1, X)$ denote the open subscheme consisting of morphisms that are birational onto their images. The (normalized) \emph{space of rational curves} $\text{RatCurves}(X)$ is defined as the quotient of the normalization $\text{Hom}^{n}_{\text{bir}}(\mathbb{P}^1, X)$ by the free action of $\text{Aut}(\mathbb{P}^1)$ via reparametrization. We then obtain a universal family
\[
\rho : \text{Univ}(X) \longrightarrow \text{RatCurves}(X),
\]
which is a $\mathbb{P}^1$-bundle, along with an evaluation map
\[
\mu : \text{Univ}(X) \longrightarrow X,
\]
such that the morphism
\[
\rho \times \mu : \text{Univ}(X) \to \text{RatCurves}(X) \times X
\]
is finite. Let $f \in \text{Hom}_{\text{bir}}(\mathbb{P}^1, X)$, and denote its image by $C \subset X$.

\begin{definition}
A rational curve $C$ on $X$ is called \emph{free} if the pullback $f^*(T_X)$ of the tangent bundle is globally generated.
\end{definition}

Every free morphism corresponds to a smooth point of $\text{Hom}_{\text{bir}}(\mathbb{P}^1, X)$, and thus also of $\text{RatCurves}(X)$. The subset of free curves forms a smooth open subscheme $\text{RatCurves}_{\text{fr}}(X)$ inside $\text{RatCurves}(X)$.

Each irreducible component $\mathcal{K}$ of $\text{RatCurves}(X)$ is a (normal) quasi-projective variety equipped with a quasi-finite morphism to the Chow variety of $X$, whose image consists of Chow points of irreducible, generically reduced rational curves. These components are called \emph{families of rational curves} on $X$. For such a family $\mathcal{K}$, we have a universal family
\[
\rho: \mathcal{U} = \rho^{-1}(\mathcal{K}) \to \mathcal{K},
\]
which is again a $\mathbb{P}^1$-bundle, and an evaluation map
\[
\mu : \mathcal{U} \to X.
\]

For any point $x \in X$, define $\mathcal{U}_x = \mu^{-1}(x)$ and $\mathcal{K}_x = \rho(\mathcal{U}_x)$. Then $\mathcal{K}_x$ parametrizes the subfamily of curves passing through $x$. The induced morphism
\[
\rho_x : \mathcal{U}_x \to \mathcal{K}_x
\]
is finite.

\begin{definition}
A family $\mathcal{K}$ is said to be \emph{covering} if $\mu$ is dominant; that is, $\mathcal{K}_x$ (equivalently, $\mathcal{U}_x$) is non-empty for a general point $x \in X$. If, in addition, $\mathcal{K}_x$ (or $\mathcal{U}_x$) is projective for general $x$, then $\mathcal{K}$ is called a \emph{family of minimal rational curves}, or simply a \emph{minimal family}. Any member of such a family is called a \emph{minimal rational curve}, or simply a \emph{minimal curve}.
\end{definition}

Recall that a \emph{very general point} of $X$ is a point lying outside a countable union of proper closed subvarieties. The following results provide criteria for the existence of free rational curves.

\begin{theorem}[{\cite[Theorem~II.3.11]{Kol96}}]\label{thm:free}
Any rational curve passing through a very general point of $X$ is free.
\end{theorem}

\begin{theorem}[{\cite[Theorem~IV.1.9]{Kol96}, \cite[Proposition~1.1]{KMM}}]
Let $\mathcal{K}$ be a family of rational curves on $X$. Then $\mathcal{K}$ is a covering family if and only if it contains a (free) curve passing through a very general point of $X$. Moreover, such a family exists if and only if $X$ is uniruled.
\end{theorem}

Following \cite{Hwa14}, we recall the notion of minimal-degree covering families.
 Let $L$ be an ample line bundle on $X$. The \emph{$L$-degree} of a family $\mathcal{K}$, denoted $\deg_L(\mathcal{K})$, is defined as the degree of $L$ on any (and hence all) members of $\mathcal{K}$. A covering family $\mathcal{K}$ is called a \emph{minimal covering family with respect to $L$} if $\deg_L(\mathcal{K})$ is minimal among all covering families on $X$. If this holds for some ample line bundle $L$, then $\mathcal{K}$ is called a \emph{minimal-degree covering family}, and any of its members is called a \emph{minimal-degree curve}.

\begin{lemma}[{\cite[Section~3]{Hwa14}}]\label{lem:mindeg}
The minimum of $\deg_L(\mathcal{K})$ over all minimal families $\mathcal{K}$ on $X$ is attained for a minimal-degree covering family.
\end{lemma}

\begin{theorem}[{\cite[Theorem~IV.2.10]{Kol96}}]
Minimal covering families exist on every uniruled variety.
\end{theorem}

\subsection{\texorpdfstring{$T$-stable curves on $\widehat{X}(\widehat{w})$}{T-stable curves on X-hat(w-hat)}}

Now we construct the $T$-stable curves on $\widehat{X}(\widehat{w})$ and describe their relationship with the $T$-stable curves on $Z(\tilde{w})$.

We begin with the following observation:

\begin{lemma}\label{lem:curves}
There is a one-to-one correspondence between the set of $T$-stable curves on $\widehat{X}(\widehat{w})$ and the set of $T$-stable curves on $Z(\tilde{w})$.
\end{lemma}

\begin{proof}
By Corollary \ref{cor:birational}, the map $\tilde{\pi}$ sends the open $B$-stable neighbourhood $B\tilde{x}$ of the base point $\tilde{x} \in Z(\tilde{w})$ isomorphically onto the open $B$-stable neighbourhood $B\widehat{x}$ of the base point $\widehat{x} \in \widehat{X}(\widehat{w})$. Therefore, $\tilde{\pi}$ and $\widehat{\pi}$ induce isomorphisms on the $T$-stable curves passing through the respective base points.
\end{proof}

We keep the notations of Section \ref{sec:basics} and Section \ref{sec:line}. We now explicitly construct the $T$-stable curves of $\widehat{X}(\widehat{w})$ passing through the base point $\widehat{x}$.

First recall the $T$-stable curves in the Schubert variety $X(w)$ and in the Bott-Samelson variety $Z(\tilde{w})$.
\begin{lemma}[\cite{BK21}, {Lemma~3.5}]\label{lem:Tstable1}
    The $T$-stable curves in $X(w)$ through the base point $wx$ are exactly the \[C_{w,\beta}:=\overline{U_{-\beta}wx},\] where $\beta\in w(R^+)\cap R^{^-}$.
\end{lemma}
As similar to Lemma \ref{lem:Tstable1}, we describe the $T$-stable curves through the base point $\tilde{x}$ in $Z(\tilde{w})$.
\begin{lemma}[\cite{BK21}, {Lemma~4.1(a)}]
The $T$-stable curves in $Z(\tilde{w})$ through $\tilde{x}$ are exactly the 
\[\tilde{C}_j:=\overline{{U}_{\beta_{j}}\tilde{x}},\]
where $1\leq j\leq m$.
\end{lemma}

Let us denote by $\widehat{C}_{k,j}$ the $T$-stable curves of $\widehat{X}(\widehat{w})$ passing through the base point $\widehat{x}$, where $1 \leq k \leq r_j$ and $1 \leq j \leq m$. Then, by the previous discussion,
\[
\widehat{C}_{k,j} {\xleftarrow[\tilde{\pi}]{\simeq}} \tilde{C}_{t_j(k)} {\xrightarrow[\pi]{\simeq}} C_{w,-\beta_{t_j(k)}} = \overline{U_{\beta_{t_j(k)}} w x},
\]
where $ t_j(k) = \sum_{i=1}^{j-1} \ell(w_i) + k $ and
\[
\beta_{t_j(k)} = \left(\prod_{i=1}^{j-1} \prod_{d=1}^{r_i} s_{d,i} \right) \left( s_{1,j} \cdots s_{k-1,j}(\alpha_{k,j}) \right),
\]
with $\alpha_{k,j}$ denoting the simple root corresponding to the simple reflection $s_{k,j}$. More explicitly, we have
\[
\widehat{C}_{k,j} = \overline{U_{\beta_{t_j(k)}} \widehat{x}}.
\]
Thus, we obtain:
\begin{proposition}\label{prop:curve} We keep the notation as above.
    The $T$-stable curves of $\widehat{X}(\widehat{w})$ passing through the base point $\widehat{x}$ are given by 
    \[
\widehat{C}_{k,j} = \overline{U_{\beta_{t_j(k)}} \widehat{x}},
\]
where $1 \leq k \leq r_j$ and $1 \leq j \leq m$.
\end{proposition}

\begin{lemma}\label{lem:free}
    For $1 \leq k \leq r_j$ and $1 \leq j \leq m$, the $T$-stable curve $
\widehat{C}_{k,j}$ is a free rational curve on $\widehat{X}(\widehat{w})$.
\end{lemma}
\begin{proof}
By Proposition \ref{prop:curve}, the curves pass through the point $\hat{x}$. The proof follows from Theorem \ref{thm:free}.
\end{proof}
\subsection{Upper bound}

We begin with the following observation:

\begin{lemma}\label{lem:pullback}
Let $\mathcal{L}$ be a line bundle on $\widehat{X}(\widehat{w})$. Then, for all $1 \leq k \leq r_j$ and $1 \leq j \leq m$, we have
\[
\mathcal{L} \cdot \widehat{C}_{k,j} = \tilde{\pi}^*\mathcal{L} \cdot \tilde{C}_{t_j(k)}.
\]
\end{lemma}

\begin{proof}
By Lemma~\ref{lem:curves}, the morphism $\tilde{\pi}$ induces an isomorphism between the $T$-stable curves $C:=\widehat{C}_{k,j} \subset \widehat{X}(\widehat{w})$ and $\widetilde{C}:=\tilde{C}_{t_j(k)} \subset Z(\tilde{w})$. We have the isomorphism \[\tilde{\pi}|_{\widetilde{C}}:\widetilde{C}\to C.\]
Then, the degree of a line bundle is preserved under pullback:
\[
\deg\left( \tilde{\pi}^*\mathcal L \big|_{\widetilde{C}} \right) = \deg\left( \mathcal L \big|_{C} \right).
\]
Hence, the result follows.
\end{proof}

We now recall a result from \cite{BK21} describing the intersection pairing between line bundles and $T$-stable curves in $Z(\tilde{w})$:

\begin{proposition}\label{prop:intersection}
Let $\mathcal{L}_{t_j(k)}$ denote the line bundle on $Z(\tilde{w})$ associated with the fundamental dominant weight ${\omega}_{t_j(k)}$. Then, for all indices $j,k,p,q$, we have
\[
\mathcal{L}_{t_p(q)} \cdot \tilde{C}_{t_j(k)} = 
\begin{cases}
0 & \text{if } t_j(k) > t_p(q), \\
\left\langle {\omega}_{{t_p(q)}},\ 
s_{{t_p(q)}} \cdots s_{{t_j(k)+1}}(\alpha_{{t_j(k)}}^\vee) \right\rangle 
& \text{if } t_j(k) \leq t_p(q).
\end{cases}
\]
\end{proposition}

Now, let $\mathcal{L}$ be a line bundle on $\widehat{X}(\widehat{w})$. By Theorem~\ref{thm:LT line bundles}, its pullback to $Z(\tilde{w})$ via $\tilde{\pi}$ can be expressed as:
\begin{equation}\label{Eq:pull}
\tilde{\pi}^*\mathcal{L} = \bigotimes_{j=1}^{m} \bigotimes_{k=1}^{r_j} \mathcal{L}_{t_j(k)}^{n_{t_j(k)}}  
\end{equation}

where $n_{t_j(k)}\in \mathbb{Z}$. Define, for each $j$ and $k$,
\begin{align}\label{Eq:ljk}
\ell_{k,j} :=\ n_{t_j(k)}~ +\ 
& \sum_{k< r\leq r_j} n_{t_j(r)}\left\langle 
{\omega}_{{t_j(r)}},\ 
s_{{t_j(r)}} \cdots s_{{t_j(k)+1}}(\alpha_{{t_j(k)}}^\vee) 
\right\rangle \nonumber \\
& + 
\sum_{j < p \leq m} \sum_{q=1}^{r_p} 
n_{t_p(q)} \left\langle 
{\omega}_{{t_p(q)}},\ 
s_{{t_p(q)}} \cdots s_{{t_j(k)+1}}(\alpha_{{t_j(k)}}^\vee) 
\right\rangle.
\end{align}

We now obtain the consequence:

\begin{proposition}\label{prop:+ve}
Let $\mathcal{L}$ be a very ample line bundle on $\widehat{X}(\widehat{w})$ whose pullback to $Z(\tilde{w})$ is given by Equation~\eqref{Eq:pull}. Then, for any $1 \leq k \leq r_j$ and $1 \leq j \leq m$, we have
\[
\mathcal{L} \cdot \widehat{C}_{k,j} = \ell_{k,j}>0.
\]
\end{proposition}

\begin{proof}
The first equality is obtained immediately from Lemma~\ref{lem:pullback} together with Proposition~\ref{prop:intersection}. Since $\mathcal{L}$ is very ample, we have $\mathcal{L} \cdot \widehat{C}_{k,j} > 0$ for all $j, k$, thus the proof follows. 
\end{proof}

\begin{remark}
   For a given very ample line bundle  $\mathcal{L}$ on $\widehat{X}(\widehat{w})$, then its pullback $\tilde{\pi}^*\mathcal{L}$ to $Z(\tilde{w})$ is globally generated, which is also nef line bundle. In general the pullback need not be a very ample line bundle, see Example \ref{eg1}.
\end{remark} 

\begin{proposition}
Let $\mathcal{L}$ be a very ample line bundle on $\widehat{X}(\widehat{w})$. Then
\begin{align*}
\min\{ \mathcal{L} \cdot C \mid C \text{ is a minimal curve in} ~ \widehat{X}(\widehat{w}) \}
= \min
\big\{ \ell_{i,j} : ~1 \leq i \leq r_j ~\text{and}~ 1 \leq j \leq m~~ \big\}.
\end{align*}
\end{proposition}

\begin{proof}
First, note that the variety $ \widehat{X}(\widehat{w}) $ is uniruled. Since the degrees of all curves in the same minimal family with respect to a given polarization are equal, we have:
{\small
\begin{align*}
\min\{ \mathcal{L} \cdot C \mid C \text{ is a minimal curve in} ~ \widehat{X}(\widehat{w})\}
&= \min \{ \deg_{\mathcal{L}}(\mathcal{K}) \mid \mathcal{K} \text{ is a minimal family on } \widehat{X}(\widehat{w}) \} \\
&= \deg_{\mathcal{L}}(\mathcal{K}'),
\end{align*}
}
where $ \mathcal{K}' $ is a minimal-degree covering family with respect to $ \mathcal{L} $ thanks to Lemma~\ref{lem:mindeg}.
Note that every minimal family on $ \widehat{X}(\widehat{w}) $ contains a $ T $-stable curve (see \cite[Section 2.3 and Proposition 4.7 (2)]{BK21}). Then it follows from Lemma~\ref{lem:free} and the definition of minimal degree that:
\begin{align*}
\deg_{\mathcal{L}}(\mathcal{K}') &= \min\{ \mathcal{L} \cdot \widehat{C}_{i,j} \mid ~1 \leq i \leq r_j~\text{and}~\ 1 \leq j \leq m \} \\
&= \min
\big\{ \ell_{i,j} : ~1 \leq i \leq r_j ~\text{and}~ 1 \leq j \leq m~~ \big\} \quad (\text{By Proposition \ref{prop:+ve}}).
\end{align*}
Hence, the result follows.
\end{proof}

In the context of projective K\"ahler manifolds which are \emph{uniruled}, i.e., the property that the manifold is swept out by rational curves, the following result provides a useful estimate for the upper bound for the Gromov width in terms of the symplectic area of any minimal curve. See \cite[Theorem 1.1]{BCF24b}.

\begin{theorem}
Let $X$ be a projective complex manifold equipped with a K\"ahler form $\omega$. Then for any minimal curve $C$ in $X$, the Gromov width satisfies
\[
w_G(X, \omega) \leq \int_C \omega.
\]
\end{theorem}

The existence of such minimal curves is guaranteed on uniruled projective manifolds, and the above inequality can thus be effectively applied in our setting.

\begin{corollary}\label{cor:upperbound}
    Let $\mathcal{L}$ be a very ample line bundle on $\widehat{X}(\widehat{w})$. Then 
\begin{align*}
w_G(\widehat{X}(\widehat{w}), \omega_{\mathcal{L}}) 
&\leq \min\{ \mathcal{L} \cdot C \mid C \text{ is a minimal curve in} ~ \widehat{X}(\widehat{w}) \} \\
&= \min
\big\{ \ell_{i,j} : ~1 \leq i \leq r_j ~\text{and}~ 1 \leq j \leq m~~ \big\}.
\end{align*}
\end{corollary}


\section{Lower bound for the Gromov width of \texorpdfstring{$\widehat{X}(\widehat{w})$}{X-hat(w-hat)}}\label{sec:LB}

\subsection{Newton-Okounkov bodies}
In this subsection, we study Newton-Okounkov bodies for generalized Bott-Samelson varieties. For the basics on Newton-Okounkov bodies, we refer the reader to \cite{KK12a}, \cite{KK12b} and \cite{HK15}.

Let $X$ be a projective variety of dimension $d$ over $\mathbb{C}$ equipped with a nef line bundle $\mathcal{L}$. Fix a valuation $v$ on $\mathbb{C}(X)$ with one-dimensional leaves, and a nonzero section $\tau \in H^0(X, \mathcal{L})$. Then $H^0(X, \mathcal{L}^{\otimes k})$ embeds into $\mathbb{C}(X)$ via
\[
\sigma \mapsto \sigma/\tau^k, \quad \text{for all } k \geq 1.
\]
Define the semigroup $S \subset \mathbb{N} \times \mathbb{Z}^d$ by
\[
S = S(v, \tau) = \bigcup_{k > 0} \left\{ (k, v(\sigma/\tau^k)) \mid \sigma \in H^0(X, \mathcal{L}^{\otimes k}) \setminus \{0\} \right\}.
\]
Let $C = C(v, \tau) \subset \mathbb{R}_{\geq 0} \times \mathbb{R}^d$ be the closed convex cone generated by $S$. The set\[
\Delta(X, \mathcal{L}, v, \tau) := \left\{ x \in \mathbb{R}^d \,\middle|\, (1, x) \in C \right\}
\]
is the \textit{Newton-Okounkov body} associated to $(X, \mathcal{L}, v, \tau)$.

\subsection{Lower bound}

We now recall the definitions of the roots $\beta_{t_j(k)}$ (from Equation \ref{eq:beta}) and the corresponding root subgroups $U_{\beta_{t_j(k)}}$ of $G$. Define
\[
U_j := \prod_{k=1}^{r_j} U_{\beta_{t_j(k)}}, \quad \text{for } 1 \leq j \leq m,
\]
so that $U_1 \times \cdots \times U_m$ forms an affine open subset of $Z(\tilde{w})$ around the base point $\tilde{x}$ via the isomorphisms
\[
U_1 \times \cdots \times U_m \simeq Bwx \simeq B\tilde{x},
\]
where $x$ is a base point of $G_{w}/(P^w \cap G_w)$. The first isomorphism is well known; see \cite[II.14.5(a)]{jantzen2003} for details. The second follows from the morphism $\pi$ (see Equation \ref{BS resoltion}) which maps the $B$-orbit $B\tilde{x}$ in $Z(\tilde{w})$ isomorphically to the $B$-orbit $Bwx$ in $X(w)$.

We identify the function field $\mathbb{C}(Z(\tilde{w})) = \mathbb{C}(U_1 \times \cdots \times U_m)$ with the rational function field
\[
\mathbb{C}(x_{t_1(1)}, \dots, x_{t_1(r_1)}, \dots, x_{t_m(1)}, \dots, x_{t_m(r_m)})
\]
via the isomorphism
\[
(x_{t_1(1)}, \dots, x_{t_m(r_m)}) \mapsto 
(\exp(x_{t_1(1)} E_{\beta_{t_1(1)}}), \dots, \exp(x_{t_m(r_m)} E_{\beta_{t_m(r_m)}})),
\]
where $E_{\beta_{t_j(k)}}$ denotes the root operator corresponding to the root $\beta_{t_j(k)}$.

Given a function
\[
f = \sum a_\mathbf{q}\, \mathbf{x}^{\mathbf{q}} \in \mathbb{C}[x_{t_1(1)}, \dots, x_{t_m(r_m)}],
\]
let $(k_{t_1(1)}, \dots, k_{t_m(r_m)})$ be the maximal tuple (with respect to a fixed total ordering) among those with $a_{\mathbf{q}} \neq 0$. Define
\[
v_\beta(f) := - (k_{t_1(1)}, \dots, k_{t_m(r_m)}).
\]
This defines a valuation map
\[
v_\beta: \mathbb{C}(x_{t_1(1)}, \dots, x_{t_m(r_m)}) \setminus \{0\} \to \mathbb{Z}^{r_1+\cdots +r_m}
\]
by
\[
v_\beta\left(\frac{f}{g}\right) := v_\beta(f) - v_\beta(g), \quad \text{for } f, g \in \mathbb{C}[x_{t_1(1)}, \dots, x_{t_m(r_m)}] \setminus \{0\}.
\]
Note that $v_\beta$ is a valuation with one-dimensional leaves.

For a  given very ample line bundle $\mathcal{L}$ on $\widehat{X}(\widehat{w})$, the pullback $\tilde{\pi}^* \mathcal{L}$ is a nef line bundle on  $Z(\tilde{w})$. 
Fix $\tau \in H^0(\widehat{X}(\widehat{w}), \mathcal{L}) \setminus \{0\}$, so $\tilde{\pi}^* \tau$ is a nonzero section in $H^0(Z(\tilde{w}), \tilde{\pi}^* \mathcal{L})$. Then

\begin{lemma}\label{lem:NObody} We have the equality:
\[
\Delta(\widehat{X}(\widehat{w}), \mathcal{L}, v_\beta, \tau)
=
\Delta(Z(\tilde{w}), \tilde{\pi}^* \mathcal{L}, v_\beta, \tilde{\pi}^* \tau).
\]
\end{lemma}

\begin{proof}
Since $\widehat{X}(\widehat{w})$ and $Z(\tilde{w})$ are birational and
\[
H^0(Z({\tilde{w}}), \tilde{\pi}^* \mathcal{L}) \simeq H^0(\widehat{X}(\widehat{w}), \mathcal{L}),
\]
the result follows.
\end{proof}


We now consider the Newton-Okounkov body of $Z(\tilde{w})$ associated with a specific section introduced below. Let $\lambda \in X(T)$ be a dominant weight. Denote by $V(\lambda)$ the simple $G$-module with highest weight $\lambda$, and let $v_\lambda \in V(\lambda)$ be a highest weight vector.
We consider the morphism induced by the very ample line bundle $\mathcal L$.
\[
\mathcal{L}_{\mathbf{i},\mathbf{n}} := \tilde{\pi}^*\mathcal{L},
\quad \text{where } \mathbf{n} = (n_{t_1(1)}, \dots, n_{t_m(r_m)}) \in \mathbb{Z}^{r_1+\cdots +r_m}_{\geq 0},
\]
given by
\[
\Psi_{\mathbf{i}, \mathbf{n}}: Z(\tilde{w}) \to 
\mathbb{P}\left(\bigotimes_{j=1}^{m} \bigotimes_{k=1}^{r_j} 
V(n_{t_j(k)}{\omega}_{{t_j(k)}})\right),
\]
\[
[p_{t_1(1)}, \dots, p_{t_m(r_m)}] \mapsto 
\left[p_{{t_1}(1)}v_{n_{t_1(1)}{\omega}_{{t_1(1)}}},\dots, \prod_{j=1}^{m} \prod_{k=1}^{r_j} p_{t_j(k)} 
v_{n_{t_m(r_m)} {\omega}_{{t_m(r_m)}}} \right].
\]

At the base point $\tilde{x}$, we have $\Psi_{\mathbf{i},\mathbf{n}}(\tilde{x}) = [v_0]$, where
\[
v_0 := 
s_{{t_1(1)}} v_{n_{t_1(1)} {\omega}_{{t_1(1)}}} 
\otimes s_{{t_1(1)}} s_{{t_1(2)}} 
v_{n_{t_1(2)} {\omega}_{{t_1(2)}}} 
\otimes \dots \otimes w 
v_{n_{t_m(r_m)} {\omega}_{{t_m(r_m)}}}.
\]
This morphism induces an isomorphism of $P_{{t_1(1)}}$-modules:
\[
\Phi_{\mathbf{i},\mathbf{n}}: V^*_{\mathbf{i},\mathbf{n}} 
\to H^0(Z(\tilde{w}), \mathcal{L}_{\mathbf{i},\mathbf{n}}),
\]
where $V_{\mathbf{i},\mathbf{n}}$ is the generalized Demazure module.
Let $f_0 \in V^*_{\mathbf{i},\mathbf{n}}$ be the dual of $v_0$, and define $\phi_0 := \Phi_{\mathbf{i},\mathbf{n}}(f_0)$. Then $\phi_0$ is a nonzero section of $H^0(Z(\tilde{w}), \mathcal{L}_{\mathbf{i},\mathbf{n}})$; see \cite[Lemma~4.1]{BCF24a}.

Then, we have the following theorem from \cite[Corollary 6.3]{Fujita}.
\begin{theorem}\label{thm:Polytope}
 The Newton-Okounkov body $\Delta(Z(\tilde{w}), \tilde{\pi}^*\mathcal{L}, v_\beta, \phi_0)$ is a rational convex polytope.
\end{theorem}

Let $\tau$ be the section in $H^0(\widehat{X}(\widehat{w}), \mathcal{L})$ corresponding to $\phi_0$. Recall the numbers $\ell_{k, j}$ from Equation~\ref{Eq:ljk}, and set
\[
\kappa := \min\left\{ \ell_{k, j} \mid 1 \leq k \leq r_j,\, 1 \leq j \leq m \right\}.
\]

\begin{theorem}\label{thm:simplex} We have 
\begin{enumerate}
    \item The Newton-Okounkov body $\Delta(\widehat{X}(\widehat{w}), \mathcal{L}, v_\beta, \tau)$ is a rational convex polytope.
    \item This polytope contains a simplex of size $\kappa$.
\end{enumerate}
\end{theorem}
\begin{proof}
(1) Follows from Theorem~\ref{thm:Polytope} together with Lemma~\ref{lem:NObody}. 

(2) By similar arguments as in Section 4.4 and  Corollary~4.12 in \cite{BCF24a}, there exists a simplex of size $\kappa$ inside $\Delta(Z(\tilde{w}), \tilde{\pi}^*\mathcal{L}, v_\beta, \phi_0)$; the claim then follows from Lemma~\ref{lem:NObody}.
\end{proof}

 \begin{remark}
     As we already noted in Example \ref{eg1}, the pullback line bundle $\tilde{\pi}^*\mathcal{L}$ on $Z(\tilde{w})$ need not be very ample in general; it is only nef. In \cite{BCF24a}, the required simplex is constructed for very ample line bundles. However, the similar arguments extend to nef line bundles as well. This extension is discussed in Appendix~\ref{f_0}.
 \end{remark}

We now use Kaveh's result \cite[Corollary 11.4]{KK19} to estimate the lower bound for the Gromov width of a symplectic manifold. Applied to our case, we obtain:

\begin{corollary}\label{cor:lowerbound}
Let $(\widehat{X}(\widehat{w}), \mathcal{L})$ be a smooth generalized Bott-Samelson variety with a very ample line bundle $\mathcal L$. Then the Gromov width is at least the supremum of the sizes of open simplices contained in the interior of the Newton-Okounkov body $\Delta(\widehat{X}(\widehat{w}), \mathcal{L}, v_\beta, \tau)$. In particular, we have $w_G(\widehat{X}(\widehat{w}), \mathcal L)\geq \kappa$.
\end{corollary}
\begin{proof} The last assertion follows from Theorem \ref{thm:simplex}.
\end{proof}

\begin{proof}[\bf{Proof of Theorem \ref{thm:main}}] 
We first consider the case of an integral K\"ahler form $\omega$ of $\widehat{X}(\widehat{w})$,
that is $\omega$ is the pullback of the Fubini-Study form on the projectivization of $H^ 0(\widehat{X}(\widehat{w}), \mathcal L)$
for some very ample line bundle $\mathcal L$ of $\widehat{X}(\widehat{w})$. We thus write $\omega=\omega_{\mathcal{L}}$. In this case, the proof follows from Corollary \ref{cor:lowerbound} together with Corollary \ref{cor:upperbound}. By the conformality property of the Gromov width, we get the result.
\end{proof}

\begin{example}\label{ex:Gromov width}
    Keep the notations same as discussed in Example~\ref{eg1}. Let $\mathcal{L}$ be the very ample line bundle over $\widehat{X}(\widehat{w})$ associated with $(1,2,3)\in \mathbb{Z}^3$, then the pull-back line bundle $\mathcal{L}_{\mathbf{i}, \mathbf{n}}$ of $\mathcal{L}$ over $Z(\tilde{w})$ corresponds to $(0,0,2,1,3)\in \mathbb{Z}^5$. After evaluating the values of $\ell_{k,j}$ for $1 \leq j \leq 2$ and $1 \leq k \leq r_j$ using Equation~\ref{Eq:ljk}, we obtain:
    \[
        \ell_{1,1} = 2, \quad \ell_{2,1} = 6, \quad \ell_{3,1} = 5, \quad \ell_{4,1} = 4, \quad \text{and} \quad \ell_{1,2} = 3.
    \] 
    Therefore, by Theorem~\ref{thm:main}, the Gromov width 
    $w_G(\widehat{X}(\widehat{w}), \mathcal{L})$ is $2$.
\end{example}


\section{Seshadri constant}\label{sec:Seshadri con}
In this section, we study upper bounds for Seshadri constants on generalized Bott-Samelson varieties.

Seshadri constants were introduced by Demailly in \cite{Dem92}, motivated by an ampleness criterion for line bundles due to Seshadri; see \cite[Theorem 7.1]{Hart70}. Since their introduction, they have become a fundamental tool in understanding the local positivity of line bundles on projective varieties. For more details on Seshadri constants, we refer to \cite{Bauer}.

\begin{theorem}[Seshadri criterion]
Let $X$ be a projective variety and let $\mathcal L$ be a line bundle on $X$.
Then $\mathcal L $ is ample if and only if there exists a constant $\varepsilon > 0$ such that,
for every point $x \in X$ and every irreducible curve $C \subset X$ passing through
$x$, the following inequality holds:
\[
\mathcal L \cdot C \ge \varepsilon\, \mathrm{mult}_x(C).
\]
\end{theorem}

This result naturally raises the question of how large such a constant
$\varepsilon$ can be. To measure this local positivity more precisely,
Demailly introduced the following notion.

\begin{definition}[Seshadri constant at a point]
Let $X$ be a projective variety and let $\mathcal L$ be a nef line bundle on $X$.
For a given point $x \in X$, the \emph{Seshadri constant} of $\mathcal L$ at $x$ is defined by
\[
\varepsilon(X, \mathcal L;x)
\;:=\;
\inf_{x \in C}
\frac{\mathcal L \cdot C}{\mathrm{mult}_x(C)},
\]
where the infimum is taken over all irreducible curves $C \subset X$ passing
through the point $x$.
\end{definition}

The following result relates to the Gromov width:

\begin{proposition}[\cite{BC01}, Proposition~6.2.1]
Let $X$ be a smooth projective variety with a very ample line bundle $\mathcal L$, and let $\omega_{\mathcal L}$ be the associated Fubini–Study form. Then for any $x \in X$,
\[
\varepsilon(X, \mathcal L, x) \leq w_G(X, \omega_{\mathcal L}).
\]
\end{proposition}

From this and Theorem~\ref{thm:main}, we deduce:

\begin{corollary}\label{Cor:Ses}
Let $\widehat{X}(\widehat{w})$ be a smooth generalized Bott-Samelson variety with a very ample line bundle $\mathcal{L}$. Then for any $x \in \widehat{X}(\widehat{w})$,
\[
\varepsilon(\widehat{X}(\widehat{w}), \mathcal{L}, x) \leq \kappa .
\]
\end{corollary}

\section{Appendix}\label{f_0}
In this section, for the reader’s convenience, we prove that the Newton-Okounkov body
\[
\Delta(Z({\tilde{w}}), \mathcal{L}_{\mathbf{i}, \mathbf{n}}, v_\beta, \phi_0)
\]
associated to a nef line bundle $\mathcal{L}_{\mathbf{i}, \mathbf{n}}$ on $X$ contains a simplex of size $\kappa$ when $\mathcal{L}_{\mathbf{i}, \mathbf{n}}$ is the pullback of a very ample line bundle from a generalized Bott-Samelson variety. Recall that this pullback is not very ample in general, but it is nef. In \cite{BCF24a}, the authors assume that the line bundle is very ample. Here, using the same arguments, we obtain the result for nef line bundle $\mathcal{L}_{\mathbf{i}, \mathbf{n}}$.

Let $Z({\tilde{w}})$ be the Bott-Samelson variety associated with the reduced (admissible generalized) decomposition ${w}= s_{i_1}\cdots s_{i_m}$ of $w\in W$.
We consider a nef line bundle 
\[ 
\mathcal{L}_{\mathbf{i},\mathbf{n}} \quad \text{with} \quad \mathbf{n} = (n_{1}, \dots ,n_m) \in \mathbb{Z}^m_{\geq 0}
\] 
The associated  morphism is given by \[ 
\Psi_{\mathbf{i}, \mathbf{n}}: {Z}({\tilde{w}}) \to \mathbb{P}\left(\bigotimes_{j=1}^{m} V(n_{j}{\omega}_{i_{{j}}})\right)
\]
\[
[p_{{1}},\dots,p_{m} ]\to[p_{1}v_{n_{1}{\omega}_{i_{1}}}, p_{1}p_{2}v_{n_{2}{\omega}_{i_{2}}},\dots, \big(\prod_{j=1}^{m}p_{j})v_{n_{m}{\omega}_{i_{m}}}],
\]
where $V(\lambda)$ denotes the simple $G$-module with highest weight $\lambda$ and $v_\lambda$ be a highest weight vector. Here $\Psi_{\mathbf{i},\mathbf{n}}(\tilde{x})=[v_0]$ with \[
v_0=s_{i_{1}}v_{n_{1}{\omega}_{i_{1}}} \otimes s_{i_{1}}s_{i_{2}}v_{n_{2}{\omega}_{i_{2}}} \otimes \dots\otimes w v_{n_{m}{\omega}_{i_{m}}}
\]
Moreover, the morphism $\Psi_{\mathbf{i},\mathbf{n}}$ induces an isomorphism of $P_{i_{1}}$-modules\[
\Phi_{\mathbf{i},\mathbf{n}}: V^*_{\mathbf{i},\mathbf{n}} \to H^{0}(Z(\tilde{w}),\mathcal{L}_{\mathbf{i},\mathbf{n}}),
\]
where $V_{\mathbf{i},\mathbf{n}}$ is the generalized Demazure module. As a complex vector space, it is generated by vectors of the form:
\begin{align*}
&F^{a_{1}}_{i_{1}}\Big( 
    v_{n_{1}{\omega}_{i_{1}}} \otimes 
    F^{a_{2}}_{i_{2}}\Big( 
        v_{n_{2}{\omega}_{i_{2}}} \otimes 
        \cdots \otimes 
        F^{a_{m-1}}_{i_{m-1}}\Big(
            v_{n_{m-1}{\omega}_{i_{m-1}}} \otimes
            F^{a_{m}}_{i_{m}} 
            v_{n_{m}{\omega}_{i_{m}}}
        \Big)\cdots
    \Big)
\Big)
\end{align*}
with $a_{j}\in \mathbb{Z}_{\geq{0}}$ and $F_{i_{j}}$ being the root operator associated with the root $-\alpha_{i_{j}}$.

Let $f_0$ denotes the dual of $v_0$ in $V^*_{\mathbf{i,n}}$. More explicitly, 
\begin{equation}\label{eq:f0}
    f_0= s_{i_{1}}v^*_{n_{1}{\omega}_{i_{1}}} \otimes s_{i_{1}}s_{i_{2}}v^*_{n_{2}{\omega}_{i_{2}}} \otimes \dots\otimes w v^*_{n_{m}{\omega}_{i_{m}}}
\end{equation}

and define \[\quad \phi_0: =\Phi_{\mathbf{i},\mathbf{n}}(f_0).
\]

\begin{lemma}\label{Section}
    The section $\phi_0\in H^0(Z({\tilde{w}}), \mathcal{L}_{\mathbf{i},\mathbf{n}})$ does not vanish on the open subset $B{\tilde{x}}$ of $Z({\tilde{w}}).$
\end{lemma}
\begin{proof}
    Note that $f_0$ does not vanish on $U_1\times\cdots \times U_m[v_0]$, by a simple consideration on weights. By the morphism $\Phi_{\mathbf{i},\mathbf{n}}$ and the isomorphism $U_1\times\cdots \times U_m \simeq B{\tilde{x}} $, the result follows.
\end{proof}
With the help of Lemma \ref{Section} we are able to define the Newton-Okounkov body associated to $\phi_0$, that is\[
\Delta_{\mathbf{i},\mathbf{n}}: = \Delta({Z}({\tilde{w}}), \mathcal{L}_{\mathbf{i},\mathbf{n}}, v_\beta, \phi_0).
\]

\begin{theorem}[\cite{Fujita},~Cor.~6.3]\label{Fujita}
The Newton-Okounkov body $\Delta_{\mathbf{i}, \mathbf{n}}$ is a convex polytope.
\end{theorem}


We now prove the following combinatorial lemma, which will be needed later.

\begin{lemma}\label{+roots}
   For every $1 \leq k<j \leq m$ with $j-k \geq 2$ and every $j<q\leq m$ the following roots are positive:
\[
s_{i_{k+1}} \cdots s_{i_{j-1}} ( \alpha_{i_{j}} ) \quad \text{and} \quad
s_{i_{q}} \cdots s_{i_{j+1}}( \alpha_{i_{j}} )
\]respectively.
\end{lemma}

\begin{proof}
  We prove the first assertion of the lemma by induction on $k$. 
Recall $\beta_j$'s from Equation \ref{eq:beta}. Suppose $k = 1$ and assume for contradiction that
\[
s_{i_2} \cdots s_{i_{j-1}}(\alpha_{i_{j}})
\]
is a negative root, say $\alpha$, for any $3\leq j\leq m$. Then applying $s_{i_{1}}$ to $\alpha$ yields
$
s_{i_1}(\alpha) = \beta_{j},
$
which is known to be a positive root. This implies $\alpha = -\alpha_{i_1}$. Consequently, we obtain $\beta_{j} = \beta_{1}$. However, by construction, all $\beta_{i}$'s are distinct, so this is a contradiction. Hence, the base case holds.  Now, assume the statement holds for some $k - 1 \geq 2$, i.e., assume that
\[
s_{i_{k}} \cdots s_{i_{j-1}}(\alpha_{i_{j}})
\]
is a positive root for all $j$ with $j-k \geq 1$. We want to prove that
\[
s_{i_{k+1}} \cdots s_{i_{j-1}}(\alpha_{i_{j}})
\]
is also a positive root for $j-k \geq 2$.

Suppose that this root is negative; we denote it by $\beta$. By the induction hypothesis, applying $s_{i_{k}}$ to $\beta$ gives a positive root. This would imply that $\beta = -\alpha_{i_{k}}$, so,
\[
\beta_{j} = s_{i_{1}}\cdots s_{i_{k}}(\beta) = s_{i_{1}}\cdots s_{i_{k}}(-\alpha_{i_k}) = \beta_{k}.
\]
again we got a contradiction as $j\neq k$.   

Now we prove the second assertion of the lemma. Assume for contradiction that \[s_{i_{q}}\cdots s_{i_{j+1}}(\alpha_{i_j})\] is a negative root. Then, we have
\begin{align*}
s_{i_{q-1}} \cdots s_{i_{j+1}}(\alpha_{i_j}) &= -\alpha_{i_q}\\ -\alpha_{i_j}&= s_{i_{j+1}}\cdots s_{i_{q-1}}(\alpha_{i_q})\\s_{i_1}\cdots s_{i_{j-1}}(\alpha_{i_j})&= s_{i_1}\cdots s_{i_{j-1}} s_{i_j}s_{i_{j+1}}\cdots s_{i_{q-1}}(\alpha_{i_q})\\
\beta_j&= \beta_q.
\end{align*}
But $j\neq q$, again we got a contradiction. Hence, the desired root must be positive.
\end{proof}

\begin{remark}
   In \cite[Lemma~4.7]{BCF24a}, the lemma is stated without the condition $j-k \geq 2$; however, the results in the article \cite{BCF24a} where this lemma is used remain unaffected.
\end{remark}

Given $\mathbf{n}\in \mathbb{Z}^m_{\geq0}$. For each $1\leq j\leq m$, we set $
\lambda_{j}:= n_{j}{\omega}_{i_j} 
$
and \[
\ell_j:= n_j + 
\sum_{j < q \leq m} 
 \langle 
\lambda_{q},\ 
s_{i_q} \cdots s_{i_{j+1}}(\alpha_{i_{j}}^\vee) 
\rangle.
\]

Now assume that all the $\ell_j$ are nonzero (which is the case in our generalized Bott-Samelson setting). Then by Lemma \ref{+roots}, all these $\ell_j$'s are positive.

   \begin{remark}
       In fact, in generalized Bott-Samleson case each $\ell_j$ for $1\leq j \leq m$ is positive, see Proposition \ref{prop:+ve}. 
   \end{remark}

Recall the definitions of the linear form $f_0$ and the morphism $\Phi_{\mathbf{i}, \mathbf{n}}$ from Equation \ref{eq:f0}. Define
\[
f_{j} := F_{\beta_{j}}^{\ell_{j}} f_0 
\quad \text{and} \quad 
\phi_{j} := \Phi_{\mathbf{i}, \mathbf{n}}(f_{j}),
\]
where $F_{\beta_{j}} = E_{-\beta_{j}}$ is the root operator corresponding to the negative root $-\beta_{j}$.

\begin{lemma}\label{zero}
    For each $j$, we have
    \[
    E_{\beta_{j}}(s_{i_{t_1(1)}} \cdots s_{i_{k}} v_{\lambda_{k}}) = 0
    \]
    for every $1 \leq k < j \leq m$. In particular, we have the following identity
          \begin{equation}\label{f_j's}
    f_{j} = 
    s_{i_{1}} v_{\lambda_{1}}^* \otimes \cdots \otimes 
    F^{\ell_{j}}_{\beta_{j}}
    \left(s_{i_{1}} \cdots s_{i_{j}} v_{\lambda_{j}}^* 
    \otimes \cdots \otimes 
    w v_{\lambda_{m}}^*\right). 
     \end{equation}
\end{lemma}
\begin{proof}
Lemma \ref{+roots} shows that
\[
\langle 
s_{i_{1}} \cdots s_{i_{k}} \lambda_{k},\ 
\beta_{j}^\vee 
\rangle = 
\langle 
\lambda_{k},\ 
s_{i_{k+1}} \cdots s_{i_{j-1}} (\alpha_{i_{j}})^\vee
\rangle \geq 0 
\quad \text{if } j-k\geq 2.
\]
If $j-k = 1$, then
\begin{align*}
\langle 
s_{i_{1}} \cdots s_{i_{k}} \lambda_{k},\ 
\beta_{{j}}^\vee 
\rangle 
&= \langle 
s_{i_{1}} \cdots s_{i_{k}} \lambda_{k},\ 
s_{i_1} \cdots s_{i_k} (\alpha_{i_{k+1}}^\vee) 
\rangle \\
&= \langle 
\lambda_{k},\ 
(\alpha_{i_{k+1}}^\vee) 
\rangle \geq 0.
\end{align*}
Thus, in either case, we conclude that
\[
E_{\beta_{j}}\left(s_{i_{1}} \cdots s_{i_{k}} v_{\lambda_{k}}\right) = 0
\]
since $s_{i_{1}} \cdots s_{i_{k}} v_{\lambda_{k}}$ is an extremal weight vector. The result then follows by duality.
\end{proof}

\medskip

Recall now the definition of the weight vector $v_0$ from the beginning of Section \ref{f_0}.

\begin{lemma}\label{nonzero}
    For every $1 \leq j \leq m$, we have:
    \[
    E_{\beta_{j}}^{\ell_{j}+1}(v_0) = 0,
    \quad \text{and} \quad
    E_{\beta_{j}}^{\ell_{j}}(v_0) \neq 0.
    \]
\end{lemma}

\begin{proof}
For $1\leq j\leq \ell\leq m$, we set
\[
a_{j \ell} := 
\langle 
s_{i_{1}} \cdots s_{i_{\ell}} \lambda_{\ell},\ 
\beta_{j}^\vee 
\rangle.
\]
Then,
\[
a_{j q} = 
\begin{cases}
    -n_{j} & \text{if } j=q, \\
    -\langle 
    \lambda_{q},\ 
    s_{i_{q}} \cdots s_{i_{j+1}}(\alpha_{i_{j}}^\vee) 
    \rangle & \text{if } j<q\leq m.
\end{cases}
\]
By Lemma \ref{+roots}, each $a_{jq} \leq 0$. Since $s_{i_{1}} \cdots s_{i_{q}} v_{\lambda_{q}}$ is an extremal weight vector, we have:
\[
E_{\beta_{j}}^{1 - a_{jq}}\left( s_{i_{1}} \cdots s_{i_{q}} v_{\lambda_{q}} \right) = 0
\]
for all $j \leq q \leq m$. Now, by definition,
\[
\ell_{j} := \sum_{q = j}^{m} -a_{jq},
\]
which implies:
\[
E_{\beta_{j}}^{\ell_j+1}
\left( 
s_{i_{1}} \cdots s_{i_{j}} v_{\lambda_{j}} \otimes \cdots \otimes w v_{\lambda_{m}} 
\right) = 0. 
\]
Finally, Lemma \ref{zero} implies the desired statement for $v_0$.
\end{proof}

\begin{lemma}\label{Section2}
The sections $\phi_{j} \in H^0(Z({\tilde{w}}), \mathcal{L}_{\mathbf{i}, \mathbf{n}})$ are nontrivial for all $1\leq j \leq m$.
\end{lemma}

\begin{proof}
Let $u_{j} = \exp(E_{\beta_{j}}) \in U_{\beta_{j}}$. By Lemmas \ref{zero} and \ref{nonzero}, we have:
\begin{align*}
u_{j}(v_0) 
&= E_{\beta_{j}}^0(v_0) + \cdots + 
\frac{E_{\beta_{j}}^r(v_0)}{r!} + \cdots + 
\frac{E_{\beta_{j}}^{\ell_{j}}(v_0)}{\ell_{j}!} \\
&= v_0 + \cdots + 
s_{i_{1}} v_{\lambda_{1}} \otimes \cdots \otimes 
s_{i_{1}} \cdots s_{i_{j-1}} v_{\lambda_{j-1}} \otimes \\
& \quad \quad\quad\quad\quad\quad \frac{E_{\beta_{j}}^r}{r!}
\left( s_{i_{1}} \cdots s_{i_{j}} v_{\lambda_{j}} 
\otimes \cdots \otimes w v_{\lambda_{m}} \right) 
+ \cdots
\end{align*}
By Equality \ref{f_j's}, we know that $f_{j}(E_{\beta_{j}}^r(v_0)) = 0$ for all $r < \ell_{j}$. Moreover, since $E_{\beta_{j}}^{\ell_{j}}(v_0) \neq 0$ by Lemma \ref{nonzero}, it follows that
\[
f_{j}(E_{\beta_j}^{\ell_{j}}(v_0)) \neq 0.
\]
Hence, we conclude that $f_{j}(u_{j}(v_0)) \neq 0$, proving that $\phi_{j} \neq 0$.
\end{proof}

\begin{proposition}\label{polynomials}
For each $ 1 \leq j \leq m $, we have
\[
\frac{\phi_{j}}{\phi_0}
= a_{j} x_{j}^{\ell_{j}}
\in \mathbb{C}[x_{1}, x_{2}, \dots, x_{m}]
\]
for some $ a_{j} \in \mathbb{C} $.
\end{proposition}

\begin{proof}
Let $\phi \in H^0(Z({\tilde{w}}), \mathcal{L}_{\mathbf{i}, \mathbf{n}})$. By Lemma \ref{Section}, the quotient $\frac{\phi}{\phi_0}$ is a regular function on $B\tilde{x}$, and hence an element of
$
\mathbb{C}[x_{1}, x_{2},\dots, x_{m}].
$
Using the same argument as in the proof of Lemma \ref{Section2}, and the fact that the roots $\beta_{j}$'s are pairwise distinct, we find that
\[
f_{j}(U_{\beta_{k}} v_0) = 0 \quad \text{for all } j \neq k.
\]
Thus, the quotient $\frac{\phi_{j}}{\phi_0}$ lies in $\mathbb{C}[x_{j}]$.
Finally, observe that
\[
f_{j}(u_{\beta_{j}} v_0) 
= f_{j}\left(\frac{(a E_{\beta_{j}})^{\ell_{j}} v_0}{\ell_{j}!}\right)
\quad \text{for some } a \in \mathbb{C},
\]
where $u_{\beta_{j}} = \exp(a E_{\beta_{j}})$. This completes the proof.
\end{proof}

\begin{corollary}
The polytope $\Delta_{\mathbf{i}, \mathbf{n}}$ contains a simplex of size
\[
k = \min \{ \ell_{j} : 1 \leq j \leq m \}.
\]
\end{corollary}

\begin{proof}
First note that $\Delta_{\mathbf{i}, \mathbf{n}}$ contains the origin, since $v_\beta(\phi_0) = 0$. From the definition of valuation $v_\beta$ and Proposition \ref{polynomials}, we obtain
\[
v_\beta(\phi_{j}) = -\left(0, \dots, 0, \ell_{j}, 0, \dots, 0\right) = -\ell_{j} e_{j}.
\]
Together with the convexity of $\Delta_{\mathbf{i}, \mathbf{n}}$ (Theorem \ref{Fujita}), this implies the corollary.
\end{proof} 


\noindent {\bf Acknowledgments.} We sincerely thank Stéphanie Cupit-Foutou for her valuable and helpful discussions.


\end{document}